\documentclass[a4paper,reqno,11pt]{amsart}

\usepackage[left=1 in, right=1 in,top=1 in, bottom=1 in]{geometry}

\usepackage{amsfonts}
\usepackage{amssymb}
\usepackage{amsthm}
\usepackage{amsmath}
\usepackage{mathrsfs}
\usepackage{color}

\usepackage{tikz} 
\pgfrealjobname{WNGD}
\usetikzlibrary{calc}
\usetikzlibrary{intersections}

\usepackage{pdfsync} 

\allowdisplaybreaks

\numberwithin{equation}{section} 

\newcommand{\TV}{{\mathrm{TV}_0^1}}
\newcommand{\TVchi}[1]{\mathrm{TV}_{\!\chi^N\!#1}}
\newcommand*\rmd{\mathop{}\!\mathrm{d}}

\newcommand{\Str}{\mathrm{Str}}

\newcommand{\mbr}{\mathbb{R}}
\newcommand{\mbs}{\mathbb{S}}
\newcommand{\mbn}{\mathbb{N}}

\newcommand{\pt}{\partial_t}
\newcommand{\px}{\partial_x}

\newcommand{\emptline}{$\phantom{A}$}

\newcommand{\dm}{\diamondsuit}
\newcommand{\dmh}{\dm^N_{1,m,n}}
\newcommand{\dml}{\dm^N_{1,m,n,L}}
\newcommand{\dmr}{\dm^N_{1,m,n,R}}
\newcommand{\alphah}{\alpha^N_{1,m,n}}
\newcommand{\alphal}{\alpha^N_{1,m,n,L}}
\newcommand{\alphar}{\alpha^N_{1,m,n,R}}
\newcommand{\betah}{\beta^N_{1,m,n}}
\newcommand{\gammah}{\gamma^N_{1,m,n}}

\newcommand{\ts}{{\sigma}}

\newcommand{\tme}{\widetilde{M}_E}
\newcommand{\vth}{\vartheta}
\newcommand{\hts}[1]{\hat{\sigma}_{#1}^N}

\newcommand{\dt}{\Delta t^N\!}
\newcommand{\dx}{\Delta x^N\!}

\newcommand{\TVm}{{\mathrm{TV}_{(m-1)\dx}^{(m+1)\dx}}}

\newcommand{\LR}{{\substack{L\\R}}}

\newcommand{\me}{M_E}
\newcommand{\ms}[1]{M_S(#1)}
\newcommand{\msn}[1]{M_S^N(#1)}
\newcommand{\mi}{\ms{0}}
\newcommand{\mz}{\ms{T_0}}
\newcommand{\mb}{M_*}

\newcommand{\smpsymb}{\bullet}

\newtheorem{lem}{Lemma}[section]
\newtheorem*{cor}{Corollary}
\newtheorem{prop}[lem]{Proposition}
\newtheorem{thm}{Theorem}
\newtheorem{fact}{Fact}
\newtheorem*{claim}{Claim}

\newtheorem*{rem}{Remark:}

\begin{document}

\title[Weakly Nonlinear Gas Dynamics]{Global Entropy Solutions to \\ Weakly Nonlinear Gas Dynamics}

\author{Peng Qu}
\address{
School of Mathematical Sciences\\ Fudan University\\
Shanghai 200433, China
\newline \indent and
\newline \indent The Institute of Mathematical Sciences\\
The Chinese University of Hong Kong\\
Shatin, NT, Hong Kong}
\email[P. Qu]{pqu@fudan.edu.cn}
\thanks{P. Qu is partially supported by  National Science Foundation of China (No.~11501121), Yang Fan Foundation of Shanghai on Science and Technology (No.~15YF1401100), Hong Kong RGC Earmarked Research Grants CUHK-14305315 and CUHK-4048/13P, Shanghai Key Laboratory for Contemporary Applied Mathematics at Fudan University and an initiative funding of Fudan University.}

\author{Zhouping Xin}
\address{The Institute of Mathematical Sciences\\
The Chinese University of Hong Kong\\
Shatin, NT, Hong Kong}
\email[Z. Xin]{zpxin@cuhk.edu.hk}
\thanks{Z. Xin is partially supported by the Zheng Ge Ru Foundation, Hong Kong RGC Earmarked Research Grants CUHK-14305315 and CUHK-4048/13P, a Focus Area Grant from the Chinese University of Hong Kong, and NSFC/RGC Joint Research Scheme N-CUHK443/14.}





\begin{abstract}
Entropy weak solutions with bounded periodic initial data are considered for the system of weakly nonlinear gas dynamics.
Through a modified Glimm scheme, 
an approximate solution sequence is constructed, 
and then a priori estimates are provided with the methods of approximate characteristics and approximate conservation laws,
which gives not only the existence and uniqueness but also the uniform total variation bounds for the entropy solutions.
\end{abstract} 

\maketitle


\section{Introduction}\label{intro sec}

Consider the following system of weakly nonlinear gas dynamics.
\begin{equation} \label{1.1_sys}
\left\{ \begin{aligned}
& \pt \sigma_1 + \frac{\alpha}{2} \px (\sigma_1^2) + \frac{\beta}{2} \int_{-1}^{1} \frac{1}{2} \sigma'_2(\frac{x+y}{2}) \sigma_3(y,t) \rmd y = 0, \\
& \pt \sigma_3 - \frac{\alpha}{2} \px (\sigma_3^2) - \frac{\beta}{2} \int_{-1}^{1} \frac{1}{2} \sigma'_2(\frac{x+y}{2}) \sigma_1(y,t) \rmd y = 0, \\
& t=0:\; \sigma_1 = \sigma_{1,0}(x), \, \sigma_3 = \sigma_{3,0}(x),
\end{aligned}\right. 
\end{equation}
where $\sigma_1 = \sigma_1 (x,t)$ and $\sigma_3 = \sigma_3 (x,t)$ are unknown functions, with $\sigma_{1,0}$ and $\sigma_{3,0}$ as given initial states satisfying 
\begin{gather}
\sigma_{1,0}(x+1) = \sigma_{1,0}(x), \quad \sigma_{3,0}(x+1) = \sigma_{3,0}(x), \label{1.2_1}\\
\int_{0}^{1} \sigma_{1,0}(x) \rmd x =0, \quad \int_{0}^{1} \sigma_{3,0}(x) \rmd x = 0. \label{1.2_2}
\end{gather}
While, $\sigma_2 = \sigma_2 (x)$ is any given $W_{\mathrm{loc}}^{1,1}$ function with 
\begin{equation} \label{1.3_sigma2}
\sigma_2(x + \frac{1}{2}) = \sigma_2 (x), 
\footnote{One may refer to \cite{Majda_1984} and \cite{Pego_WNLO__1988}, which show the reason why one needs only to count the even modes of $\sigma_2$.} \quad 
\int_0^{\frac12} |\sigma'_2(x)| \rmd x = \me \leq \infty,
\end{equation}
and $\alpha, \, \beta$ are given positive constants.

This system is first derived by A.~Majda and R.~Rosales in \cite{Majda_1984} from the $1$-dimensional full Euler equations with \emph{periodic} initial data
through the method of weakly nonlinear geometric optics approximation
to study the behavior of the solutions especially for the cases with resonance effects.
Different from the Cauchy problem with initial data of small total variation,
which has a quite complete theory for existence \cite{Glimm_basic_1965}, \cite{smoller_book_1983} and uniqueness \cite{Bressan_book_2000},
most aspects of the Cauchy problem for quasilinear systems of hyperbolic conservation laws with small periodic initial data are still open.
One of the main difficulties may lie on the fact that the periodicity prevent the waves from separation and thus the system does not possess a decreasing Glimm functional to control the nonlinear effects,
and almost all the classical analysis methods fail in this case.
The celebrated work of J.~Glimm--P.D.~Lax \cite{Glimm_Lax_1970} shows that for the isentropic Euler equations the Cauchy problem with small periodic initial data admits a global entropy solution.
The method of their work is to use the cancellations occurring for genuinely nonlinear characteristics when intrafamily shock-rarefaction waves coalesce.
This cancellation is well analyzed with the method of approximate characteristics and approximate conservation laws, 
which, combined with the nature of the system that the isentropic Euler system has a complete set of Riemann invariant coordinates and thus a relatively weak interfamily  nonlinear interaction effects,
provide a $t^{-1}$ decay for the total variation per period of the solution, as well as the global existence.
This result is then further developed and generalized by many works, among them are \cite{Dafermos_asymp_periodic_1995}, \cite{Dafermos_book_2005}, \cite{Bia2010Linfty}. 
On the other hand, for the $1$-dimensional full Euler system, 
there is one more family of characteristics which are linearly degenerate,
and the system generally does not have a complete set of Riemann invariant coordinates.
It is believed that the process that the right sound waves would be affected when the left sound waves interact with the entropy waves and vice vesa would cause the effect of nonlinear resonance for space periodic data, 
and change completely many aspects of the behavior of the solutions such as the time asymptotic.
One may refer to \cite{temple_Young2009time_periodic}, \cite{temple2010liapunov}, \cite{temple2011time}, \cite{temple2015nash} for B.~Temple and R.~Young's ongoing project to construct non-trivial time periodic and thus shock-free solutions.
Meanwhile, the effect of resonance causes huge difficulties in analysis,
and one cannot expect the cancellation effect given in \cite{Glimm_Lax_1970} originally for the isentropic case dominates all the time and the system may not undergo a strong enough decay to guarantee the global existence.
Therefore, the problem of global existence for the solutions to the full Euler system with small periodic initial data is still open.
One may refer to \cite{temple_Young1996large}, \cite{Q_Xin2015} for long time existence of entropy solutions, \cite{Chen_Da_van_vis_1995} for global existence of entropy solution for special systems, and \cite{LeFloch_Xin_class}, \cite{Li_class_1996}, \cite{Xiao_class} for the blowup result of classical solutions.

To get a better understanding of the resonance effects, 
\cite{Majda_1984} performs weakly nonlinear geometric optics approximation for general systems of hyperbolic conservation laws and provides a detailed analysis on the occurrence of the resonance, which shows that the Cauchy problems with initial data of small total variation on $\mathbb{R}$ and the systems with a complete set of Riemann invariant coordinates such as the isetropic Euler system, do not possess resonance.
Meanwhile, as one of the main objects of \cite{Majda_1984}, the full Euler system with small periodic initial data gives the system of weakly nonlinear gas dynamics \eqref{1.1_sys},
which does show the resonance feature of the full Euler system through the nonlocal interaction terms.
Later, in P.L.~Pego's remarkable work \cite{Pego_WNLO__1988}, for the system \eqref{1.1_sys} with periodic initial data,
a series of non-trivial time periodic solutions are constructed,
which makes it clear that system \eqref{1.1_sys} do possess strong resonance and its solutions have complicated behavior.
On the other hand, for general systems of weakly nonlinear geometric optics, under the assumption of genuine nonlinearity and a structure requirement on interactions, \cite{cheverry1996modulation} proves global existence of entropy solutions.
Unfortunately, the most important resonant case, namely the system \eqref{1.1_sys} of weakly nonlinear gas dynamics, is not included due to the linear degeneracy of the entropy wave.
See also \cite{MRS1988canonical} for detailed analysis on the behavior of the solutions.
One may also refer to \cite{Diperna1985validity}, \cite{joly1993resonant}, \cite{schochet1994resonant}, \cite{cheverry1997justification}, \cite{chen2013weakly} and the references therein for the justification of the weakly nonlinear geometric optics approximation, and \cite{hunter1986resonantly}, \cite{chen2006validity} for the related results of the multidimensional case.


Since the $C^1$ classical solutions to \eqref{1.1_sys} would generally blow up in finite time (see Appendix A),
it is natural to look for the entropy weak solution in this paper.
The main purpose of this paper is to get the global in time entropy weak solutions for any periodic initial data with bounded total variation over each period
\begin{equation} \label{1.4_ini}
\TV \sigma_{1,0} + \TV \sigma_{3,0} = \mi < +\infty,
\end{equation}
and obtain some uniform a priori estimates for the solutions,
where $\TV f$ denotes the total variation of a spatially periodic function $f$, with period $1$, over each of its period.
Here a global entropy solution means a solution in the sense of distribution
\begin{multline} \label{1.5}
\int_{0}^{+\infty} \int_0^1 \Big( \sigma_1 \pt \varphi_1 + \frac{\alpha}{2} \sigma_1^2 \px \varphi_1 - \frac{\beta}{2} \varphi_1 \Big( \int_{-1}^{1} \frac{1}{2} \sigma'_2(\frac{x+y}{2}) \sigma_3(y,t) \rmd y \Big) \\
+ \sigma_3 \pt \varphi_3 - \frac{\alpha}{2} \sigma_3^2 \px \varphi_3 + \frac{\beta}{2} \varphi_3 \Big( \int_{-1}^{1} \frac{1}{2} \sigma'_2(\frac{x+y}{2}) \sigma_1(y,t) \rmd y \Big) \Big) \rmd x \rmd t \\
+ \int_{0}^1 \Big( \sigma_{1,0}(x) \varphi_1(x,0) + \sigma_{3,0}(x) \varphi_3(x,0) \Big) \rmd x = 0, \qquad \forall\, \varphi_1, \varphi_3 \in C_0^1 (\mbs^1 \times \overline{\mbr^+}),
\end{multline}
which is further requested to satisfy the entropy condition that for any convex entropy function $\eta = \eta(\sigma_1, \sigma_3)$ with entropy flux function $ \psi = \psi(\sigma_1,\sigma_3)  $ satisfying 
\begin{equation}
\left\{ \begin{aligned}
\alpha \sigma_1 \partial_{\sigma_1} \eta = \partial_{\sigma_1} \psi, \\
- \alpha \sigma_3 \partial_{\sigma_3} \eta = \partial_{\sigma_3} \psi, 
\end{aligned} \right.
\end{equation}
it holds that
\begin{multline} \label{1.8_entropy_eq}
\int_{0}^{+\infty} \int_0^1 \bigg(  \eta \pt \varphi + \psi \px \varphi \\
- \frac{\beta}{2} \varphi \Big( \partial_{\sigma_1} \eta \int_{-1}^{1} \frac{1}{2} \sigma'_2(\frac{x+y}{2}) \sigma_3(y,t) \rmd y - \partial_{\sigma_3} \eta \int_{-1}^{1} \frac{1}{2} \sigma'_2(\frac{x+y}{2}) \sigma_1(y,t) \rmd y \Big) \bigg) \rmd x \rmd t\\
+ \int_{0}^{1} \eta(\sigma_{1,0},\sigma_{3,0})(x) \varphi(x,0) \rmd x \geq 0, \qquad \forall\, \varphi \in C_0^1 (\mbs^1 \times \overline{\mbr^+}).
\end{multline}

The main result in this paper is 
\begin{thm} \label{thm:main}
For any given initial data $\sigma_{1,0}$ and $\sigma_{3,0}$ satisfying \eqref{1.2_1}--\eqref{1.2_2} and \eqref{1.4_ini}, the Cauchy problem \eqref{1.1_sys} with \eqref{1.3_sigma2} admits a global entropy weak solution $\sigma = (\sigma_1,\sigma_3)^T(x,t)$, which satisfies further
\begin{equation}
	\ms{t} \overset{\mathrm{def.}}{=} \TV \sigma_1(\cdot,t) + \TV \sigma_3(\cdot,t) \leq \mb, \quad \mathrm{a.e.} \; t \in \mbr^+,
\end{equation}
where
\begin{equation} \label{1.10_mb}
\mb = \max \left\{ \frac{5}{4} \mi,  300 \frac{\beta}{\alpha} \me \right\}.\footnote{Here the number $300$ is far from being sharp and can be changed in to ``large enough constants''. It is specificaly given here, to simplyfy the naration in what follows.}
\end{equation}
Moreover, for each $T>0$ this solution is unique in the class of periodic $C(0,T;L^1)$ entropy solutions.
\end{thm}

The main idea to prove this result can be summarized as follows.
First, although, some adaptions on an approximate scheme should be made to deal with the difficulty caused by their nonlocal property,
the interaction terms in the system \eqref{1.1_sys} are linear,
since the entropy wave $\sigma_2$ is a given function in this model.
Thus, it is expected that in the worst scenario it can only cause an exponential increase for some suitable norm of the solution,
namely, the time span required for the solution to double its norm  can be bounded by $\sigma_2$ and the parameters $\alpha, \beta$,
which is independent of the solution itself.
On the other hand, the quasilinear leading terms in the system \eqref{1.1_sys} are just two decoupled inviscid Burgers equations, 
for which one cannot expect anything better to apply the methods in \cite{Glimm_Lax_1970} to get a decay, 
with the property that the solution would undergo a faster decay and require less time to halve its norm once its norm is bigger.
Therefore, to combine these two effects together, the decay effect would dominate once the sound waves $(\sigma_1,\sigma_3)^T$ are relatively stronger than the entropy wave $\sigma_2$,
which can provide the desired uniform a priori bounds for the solution.

\begin{rem}
Using the above intuition that the cancellation effects can dominate an exponential growth of the solution,
one can apply a similar procedure as this paper to show the global existence and the uniform a priori estimates for the entropy solutions to the system
\begin{equation*}
\left\{ \begin{aligned}
& \pt \sigma_1 + \frac{\alpha}{2} \px (\sigma_1^2) + \frac{\beta}{2} \int_{-1}^{1} \frac{1}{2} \sigma'_2(\frac{x+y}{2}) \sigma_3(y,t) \rmd y = B_{11} \ts_1 + B_{13} \ts_3, \\
& \pt \sigma_3 - \frac{\alpha}{2} \px (\sigma_3^2) - \frac{\beta}{2} \int_{-1}^{1} \frac{1}{2} \sigma'_2(\frac{x+y}{2}) \sigma_1(y,t) \rmd y = B_{31} \ts_1 + B_{33} \ts_3, \end{aligned}\right. 
\end{equation*}
where $B = \begin{pmatrix} B_{11} & B_{13} \\ B_{31} & B_{33} \end{pmatrix}$ is any given constant matrix, which can be chosen to model the damping or rotation effects to the system.
\end{rem}

This paper is arranged as follows.
In Section 2, the approximate scheme is introduced and the corresponding consistence result is proved under some a priori assumptions on the uniform bounds.
In Section 3, the increase of the total variation is estimated for the approximate solutions.
Then in Section 4, the methods of approximate characteristics and approximate conservation laws are applied to get the decay property for the solutions and complete the proof of Theorem \ref{thm:main}.
Appendix A is devoted to the proof of the finite time blowup for the classical solutions, and Appendices B--C provide some details of the proof.

\section{An Approximate Scheme}

In order to construct a sequence of approximate solutions, 
one may adapt J.~Glimm's celebrated approximate scheme originally derived in \cite{Glimm_basic_1965}.
Moreover, one may use the fractional step methods, such as the one developed in \cite{Dafermos_Hsiao} for hyperbolic balance laws,
and modify it to deal with the nonlocal interaction terms in the system \eqref{1.1_sys}.
 
For each $N \gg 1$, set $\dx = \frac{1}{2^N}$ as the spatial mesh length, and $\dt = \Lambda^{-1} \dx$ as the corresponding time mesh length.
Here $\Lambda$ is set to satisfy
the Courant--Friedrichs--Lewy (C.F.L. for short) condition
\[
\Lambda > \alpha \max \{\|\ts_1^N\|_{L^\infty},\|\ts_3^N\|_{L^\infty}\}.
\]
In fact, for the system \eqref{1.1_sys}, it can be verified that if one chooses
\begin{equation} \label{2.6_Lam}
\Lambda > 2 \alpha (\mb + 2)
\end{equation}
a priori, then for large enough $N$, the C.F.L. condition holds.

Next, let $\vartheta = \{\vth_n\}_{n=0}^\infty$ be a sequence of independent random variables, which is equidistributed over $[-1,1)$.
Then one can set 
\[
\ts_1^N(x,0-) = \sigma_{1,0}(x), \quad \ts_3^N(x,0-) = \sigma_{3,0}(x)
\]
to initiate the construction.

Inductively, if the approximate solution $(\ts_1^N, \ts_3^N)^T$ has been constructed for $t < n \dt\, (n \in \mbn)$, one may use the random sampling
\begin{equation} \label{2.7_hts}
\begin{aligned}
\hts{1,m,n} & = \ts_1^N((m+ \vth_n) \dx,n \dt-), \\
\hts{3,m,n} & = \ts_3^N((m+ \vth_n) \dx,n \dt-),
\end{aligned} \qquad \text{for} \; (m+n) \; \text{odd},
\end{equation}
to get the corresponding piece-wise constant function
\begin{equation*}
\begin{aligned}
\hts{1,n}(x) = \hts{1,m,n}, \qquad \forall\, x \in [(m-1)\dx,(m+1) \dx),\\
\hts{3,n}(x) = \hts{3,m,n}, \qquad \forall\, x \in [(m-1)\dx,(m+1) \dx),
\end{aligned} \qquad \text{for}\; (m+n) \; \text{odd.}
\end{equation*}
Then set
\begin{equation} \label{2.8_g}
\begin{aligned}
g_{1,m,n}^N & = \int_{-1}^{1} K(m\dx+y) \hts{3,n}(y) \rmd y = \sum_{\substack{-2^N < \tilde{m}\leq 2^N \\ \tilde{m} + n \text{ odd}}}^{} K_{m+\tilde{m}}^N \hts{3,\tilde{m},n},\\
g_{3,m,n}^N & = - \int_{-1}^{1} K(m\dx+y) \hts{1,n}(y) \rmd y = - \sum_{\substack{-2^N < \tilde{m}\leq 2^N \\ \tilde{m} + n \text{ odd}}}^{} K_{m+\tilde{m}}^N \hts{1,\tilde{m},n},
\end{aligned} \qquad \text{for} \; (m+n) \text{odd,}
\end{equation}
where
\begin{equation} 
K(x)  = \frac{\beta}{4} \sigma'_2(\frac{x}{2}) \label{2.3_rs2} 
\end{equation}
with properties
\begin{equation} \label{2.5_K}
K(x + 1) = K(x), \quad \int_{0}^{1} K(y) \rmd y = 0, \quad \| K(x)\|_{L^1[0,1]} = \frac{\beta}{4} \me 
\end{equation}
and
\begin{equation} 
K_m^N = \int_{-\dx}^{\dx} K(m\dx+y) \rmd y
\end{equation}
with properties
\begin{equation} \label{2.8_kmn_pro}
K_{m+2^N}^N = K_m^N, \quad \sum_{\substack{1 \leq m \leq 2^N \\ m+n \text{ odd}}}^{} K_m^N = 0, \quad  \sum_{\substack{1 \leq m \leq 2^N \\ m+n \text{ odd}}}^{} | K_m^N | = \frac{\beta}{4} \me.
\end{equation}
And one can define
\begin{equation} \label{2.9_tso}
\begin{aligned}
\ts_{1,m,n}^N & = \hts{1,m,n} - g_{1,m,n}^N \dt,\\
\ts_{3,m,n}^N & = \hts{3,m,n} - g_{3,m,n}^N \dt,
\end{aligned} \qquad \text{for} \; (m+n) \; \text{odd.}
\end{equation}
Then, one can solve finitely many Riemann problems on each period for two decoupled Burgers equations with $(m \dx, n \dt)$ ($(m+n)$ even) as the centers:
\begin{equation} \label{2.10_Burger}
\left\{
\begin{aligned}
& \pt \ts_1^N + \frac{\alpha}{2} \px(\ts_1^N)^2 = 0, \\
& \pt \ts_3^N - \frac{\alpha}{2} \px(\ts_3^N)^2 = 0, \\
& t=n \dt: \; \ts_1^N = \ts_{1,m,n}^N, \, \ts_3^N = \ts_{3,m,n}^N, \quad \text{for} \; x \in [(m-1)\dx, (m+1) \dx),\; (m+n) \; \text{odd.}
\end{aligned}
\right.
\end{equation}
If the C.F.L. condition holds, no wave interacts for $t \in [n \dt, (n+1) \dt)$, and one may use these Riemann solvers $(\ts_1^N,\ts_3^N)^T$ as the approximate solution on $t \in [n \dt, (n+1) \dt)$ (See Figure \ref{fig:1_scheme}).
Then one can repeat this procedure on $t \in [(n+1)\dt,(n+2)\dt)$.

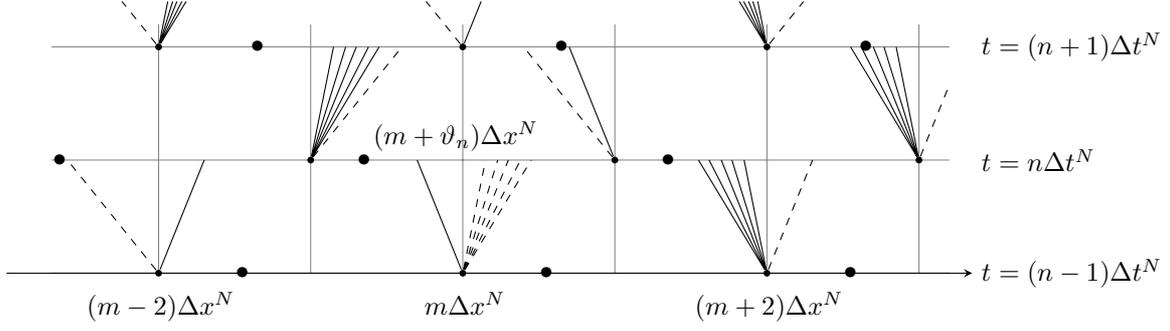
\begin{figure}
\centering
\begin{tikzpicture}[scale=1,>=stealth]
\small 

\def\xspace{2}
\def\yspace{1.5}
\def\level{2}
\def\waveno{3}

\begin{scope}
\clip (0.6,-0.4) rectangle (\waveno*\xspace*2+0.4,\level*\yspace+0.6);

\draw[help lines, xstep=\xspace ,ystep=\yspace ] (0,0) grid (\waveno*\xspace*2+0.4,\level*\yspace+0.3);

\foreach \x in {1, ..., \waveno}
{
	\node[fill,inner sep =0.9pt, shape = circle] (O1\x) at  (\x*2*\xspace-\xspace,0) {};
	\node[fill,inner sep =0.9pt, shape = circle] (O2\x) at  (\x*2*\xspace,\yspace) {};
	\node[fill,inner sep =0.9pt, shape = circle] (O3\x) at  (\x*2*\xspace-\xspace,2*\yspace) {};
}

\def\smpone{0.55} 
\def\smptwo{0.35}
\def\smpthree{0.65}
\foreach \x in {1, ..., \waveno}
{
	\node (Oth1\x) at  (\x*2*\xspace-\xspace+\smpone*\xspace,0) {$\smpsymb$};
	\node (Oth2\x) at  (\x*2*\xspace-2*\xspace+\smptwo*\xspace,\yspace) {$\smpsymb$};
	\node (Oth3\x) at  (\x*2*\xspace-\xspace+\smpthree*\xspace,2*\yspace) {$\smpsymb$};
}

\coordinate[label=above right:{$(m+\vth_n) \dx$}] (N1) at (2*\xspace+\smptwo*\xspace,\yspace);

\node (O21l) at (O21) {};
\node (O22l) at (O22) {};
\node (O23l) at (O23) {};

\draw[dashed] (O11) -- +(-0.6*\xspace,\yspace); \draw (O11) -- +(0.3*\xspace,\yspace);
\draw (O12) -- +(-0.3*\xspace,\yspace); \draw[dashed] (O12) -- +(0.3*\xspace,\yspace) (O12) -- +(0.15*\xspace,\yspace) (O12) -- +(0.23*\xspace,\yspace) (O12) -- +(0.38*\xspace,\yspace) (O12) -- +(0.45*\xspace,\yspace);
\draw (O13) -- +(-0.3*\xspace,\yspace) (O13) -- +(-0.15*\xspace,\yspace) (O13) -- +(-0.23*\xspace,\yspace) (O13) -- +(-0.38*\xspace,\yspace) (O13) -- +(-0.45*\xspace,\yspace); \draw[dashed] (O13) -- +(0.3*\xspace,\yspace);

\draw[dashed] (O22) -- +(-0.6*\xspace,\yspace); \draw (O22) -- +(-0.3*\xspace,\yspace);
\draw[dashed] (O21) -- +(0.6*\xspace,\yspace); \draw (O21) -- +(0.3*\xspace,\yspace) (O21) -- +(0.15*\xspace,\yspace) (O21) -- +(0.23*\xspace,\yspace) (O21) -- +(0.38*\xspace,\yspace) (O21) -- +(0.45*\xspace,\yspace);
\draw (O23) -- +(-0.3*\xspace,\yspace) (O23) -- +(-0.15*\xspace,\yspace) (O23) -- +(-0.23*\xspace,\yspace) (O23) -- +(-0.38*\xspace,\yspace) (O23) -- +(-0.45*\xspace,\yspace); \draw[dashed] (O23) -- +(0.3*\xspace,\yspace);

\draw[dashed] (O32) -- +(-0.6*\xspace,\yspace); \draw (O32) -- +(0.3*\xspace,\yspace);
\draw[dashed] (O31) -- +(-0.6*\xspace,\yspace); \draw (O31) -- +(0.3*\xspace,\yspace) (O31) -- +(0.15*\xspace,\yspace) (O31) -- +(0.23*\xspace,\yspace) (O31) -- +(0.38*\xspace,\yspace) (O31) -- +(0.45*\xspace,\yspace);
\draw (O33) -- +(-0.3*\xspace,\yspace) (O33) -- +(-0.15*\xspace,\yspace) (O33) -- +(-0.23*\xspace,\yspace) (O33) -- +(-0.38*\xspace,\yspace) (O33) -- +(-0.45*\xspace,\yspace); \draw[dashed] (O33) -- +(0.6*\xspace,\yspace);

\end{scope}

\draw[->] (0,0) -- (\waveno*\xspace*2+0.7,0) node[pos=1,right] {$t=(n-1)\dt$};
\path (0,\yspace) -- (\waveno*\xspace*2+0.7,\yspace) node[pos=1,right] {$t=n \dt$};
\path (0,2*\yspace) -- (\waveno*\xspace*2+0.7,2*\yspace) node[pos=1,right] {$t=(n+1) \dt$};

\node[label=below:{$(m-2) \dx$}] (O11l) at (O11) {};
\node[label=below:{$m \dx$}] (O12l) at (O12) {};
\node[label=below:{$(m+2) \dx$}] (O13l) at (O13) {};

\end{tikzpicture}
\caption{Modification of Glimm's random choice scheme}
\label{fig:1_scheme}
\end{figure}

Apparently, the approximate solutions constructed above are spatially periodic 
\begin{equation}
\ts_1^N(x+1,t) = \ts_1^N(x,t), \qquad \ts_3^N(x+1,t) = \ts_3^N(x,t).
\end{equation}
Since $\ts_i^N(x,t) \, (i=1,3)$ are piece-wise smooth, 
one may define its value to be the up right limit on the discontinuous points.

Now, the consistency of this scheme can be shown in the following
\begin{prop}\label{prop:cons}
Assume that on $(x,t) \in [0,1) \times [0,T]$, there exists a sequence of approximate solutions $(\ts_1^N, \ts_3^N)^T$ satisfying
\begin{gather}
\| \ts_1^N \|_{L^\infty} + \| \ts_3^N \|_{L^\infty} \leq C_1 < \frac{\Lambda}{\alpha}, \label{2.11_asmp_linf}\\
\ms{t} = \sup_N \msn{t} = \sup_N \{ \TV \ts_1^N(\cdot,t) + \TV \ts_3^N(\cdot,t) \} \leq C_2, \label{2.12_asmp_tv}\\
\| \ts_1^N (\cdot,t_1) - \ts_1^N (\cdot,t_2) \|_{L^1[0,1)} + \| \ts_3^N (\cdot,t_1) - \ts_3^N (\cdot,t_2) \|_{L^1[0,1)} \leq C_3 (|t_2 - t_1| + 2 \dt). \label{2.13_asmp_l1dis}
\end{gather}
Then there exists a subsequence\footnote{ For notation simplicity, we always use $(\ts_1^N, \ts_3^N)^T$ to denote the sequence of approximate solutions as well as its subsequence.}, such that for almost all choice of $\vth$, it holds
\[
(\ts_1^N, \ts_3^N)^T  \overset{L^1}{\longrightarrow} (\ts_1,\ts_3)^T,
\]
where $(\ts_1,\ts_3)^T$ is an entropy weak solution to the Cauchy problem \eqref{1.1_sys} which is periodic with zero means, i.e.,
\begin{gather}
\ts_1(x+1,t) = \ts_1(x,t), \quad \ts_3(x+1,t) = \ts_3(x,t), \\
\int_{0}^{1} \ts_1(x,t) \rmd x = \int_{0}^{1} \ts_3(x,t) \rmd x = 0.
\end{gather}
\end{prop}
\begin{proof}
The convergence is a direct application of Helly's selection principle.
Thus, it suffices to prove the consistency.

Since $\ts_1^N$ and $\ts_3^N$ are local Riemann solvers to \eqref{2.10_Burger} on each time interval $[n \dt,(n+1)\dt )$, they satisfy \eqref{2.10_Burger} piece-wisely and the Rankine-Hugoniot condition holds on each shock discontinuity.
By multiplying  the test functions $\varphi_1$ and $\varphi_3$ to \eqref{2.10_Burger} and integrating by parts, one can get
\begin{align*}
0 = & \int_{0}^{T} \int_{0}^{1} \Big( - \ts_1^N \pt \varphi_1 - \frac{\alpha}{2} (\ts_1^N)^2 \px \varphi_1 \Big) \rmd x \rmd t + \int_0^T \int_{0}^{1} \Big( - \ts_3^N \pt \varphi_3 + \frac{\alpha}{2} (\ts_3^N)^2 \px \varphi_3 \Big) \rmd x \rmd t \\
& +\sum_{n=0}^{\infty} \int_{0}^{1} \Big( - \ts_1^N (x, n \dt+) + \ts_1^N  (x, n \dt-) \Big) \varphi_1(x,n\dt) \rmd x \\
& + \sum_{n=0}^{\infty} \int_{0}^{1} \Big( - \ts_3^N (x, n \dt+) + \ts_3^N (x, n \dt-) \Big) \varphi_3(x,n\dt) \rmd x \\
& - \int_{0}^{1} \Big( \sigma_{1,0}(x) \varphi_1(x,0) + \sigma_{3,0} (x) \varphi_3(x,0) \Big) \rmd x.
\end{align*}
Thus,
\begin{align*}
& \int_0^T \int_{0}^{1} \bigg( \ts_1^N \pt \varphi_1 + \frac{\alpha}{2} (\sigma_1^N)^2 \px \varphi_1 + \ts_3^N \pt \varphi_3 - \frac{\alpha}{2} (\sigma_3^N)^2 \px \varphi_3 \\
& - \varphi_1 \Big( \int_{-1}^{1} K(x+y) \ts_3^N(y,t) \rmd y \Big) 
+ \varphi_3 \Big( \int_{-1}^{1} K(x+y) \ts_1^N(y,t) \rmd y \Big) \bigg) \rmd x \rmd t \\
& + \int_{0}^{1} \Big( \sigma_{1,0}(x) \varphi_1(x,0) + \sigma_{3,0}(x) \varphi_3(x,0) \Big) \rmd x \\
= & - \sum_{n=0}^{\infty} \int_{0}^{1} \Big( \ts_1^N (x,n \dt+) - \ts_1^N (x,n \dt-) \Big) \varphi_1(x,n\dt) \rmd x \\
& - \sum_{n=0}^{\infty} \int_{0}^{1} \Big( \ts_3^N (x,n \dt+) - \ts_3^N (x,n \dt-) \Big) \varphi_3(x,n\dt) \rmd x \\
& - \int_{0}^{T} \int_{0}^{1} \varphi_1 \Big( \int_{-1}^{1} K(x+y) \ts_3^N (y,t) \rmd y \Big) \rmd x \rmd t \\
& + \int_{0}^{T} \int_{0}^{1} \varphi_3 \Big( \int_{-1}^{1} K(x+y) \ts_1^N (y,t) \rmd y \Big) \rmd x \rmd t 
\end{align*}
Due to \eqref{2.7_hts}--\eqref{2.10_Burger}, it holds
\begin{align*}
& \ts_1^N(x, n \dt+) - \ts_1^N(x, n\dt-) \\
= & \ts_{1,m,n}^N - \ts_1^N(x, n \dt-) \\
= & (\ts_{1,m,n}^N - \hts{1,m,n}) + \big(\hts{1,m,n} - \ts_1^N(x, n\dt-) \big)\\
= & - g_{1,m,n}^N \dt + \Big( \ts_1^N((m+\vartheta_n)\dx,n \dt- ) - \ts_1^N(x, n \dt-) \Big) \\
& \qquad \qquad \forall\, x \in [(m-1)\dx, (m+1) \dx) \text{ with } m+n \text{ odd.}
\end{align*}
Similar as in \cite{Glimm_basic_1965} and \cite{smoller_book_1983}, using Glimm's random choice method,
\[
\sum_{n=0}^{\infty} \sum_{\substack{1 \leq m \leq 2^N \\ m+n \text{ odd}}}^{} \int_{(m-1)\dx}^{(m+1)\dx} \varphi_1(x,n \dt) \Big( \ts_1^N((m+\vartheta_n)\dx, n\dt-) - \ts_1^N(x, n\dt-) \Big) \rmd x \to 0, \quad \mathrm{a.s.}
\]
Meanwhile, By \eqref{2.13_asmp_l1dis},
\begin{align*}
& \left| \int_{0}^{T} \int_{0}^{1} \varphi_1(x,t) \Big( \int_{-1}^{1} K(x+y) \big( \ts_3^N(y,t) - \ts_3^N(y, [\frac{t}{\dt}] \dt-) \big) \rmd y \Big) \rmd x \rmd t \right| \\
\leq & \| \varphi_1(x,t) \|_{L^1_tL^\infty_x} \| K \|_{L^1} \max_{t} \| \ts_3^N(y,t) - \ts_3^N(y,[\frac{t}{\dt}] \dt -) \|_{L^1_{y}} \overset{N \to \infty }{\longrightarrow} 0 .
\end{align*}
while, by the random sampling method,
\begin{align*}
& \left|  \int_{0}^{T} \int_{0}^{1} \varphi_1(x,t) \Big( \int_{-1}^{1} K(x+y) \big( \ts_3^N (y,[\frac{t}{\dt}] \dt-) - \hts{3,[\frac{t}{\dt}]} (y) \big) \rmd y \Big) \rmd x \rmd t \right| \\
= & \left| \int_{0}^{T} \int_{-1}^{1} \Big( \int_{0}^{1} \varphi_1(x,t) K(x+y) \rmd x \Big) \big( \ts_3^N (y,[\frac{t}{\dt}] \dt-) - \hts{3,[\frac{t}{\dt}]} (y) \big) \rmd y \rmd t \right| \\
\to & 0, \qquad \mathrm{a.s.}
\end{align*}
By a direct decomposition, 
\begin{align*}
& \int_0^T \int_0^1 \varphi_1(x,t) \Big( \int_{-1}^{1} K(x+y) \hts{3,[\frac{t}{\dt}]}(y) \rmd y \Big) \rmd x \rmd t \\
= & \sum_{n=0}^{\infty} \sum_{\substack{1 \leq m \leq 2^N \\ m+n \text{ odd}}}  \int_{n \dt}^{(n+1)\dt} \int_{(m-1)\dx}^{(m+1)\dx} \varphi_1(x,t) \Big( \int_{-1}^{1} K(x+y) \hts{3,n}(y) \rmd y \Big) \rmd x \rmd t.
\end{align*}
then since $K(m\dx + y)$ is a piece-wise constant approximation of $K(x+y)$ and due to \eqref{2.8_g}
\begin{align*}
& \left| \sum_{n=0}^{\infty} \sum_{\substack{1 \leq m \leq 2^N \\ m+n \text{ odd}}}  \int_{n \dt}^{(n+1)\dt} \int_{(m-1)\dx}^{(m+1)\dx} \varphi_1(x,t) \Big( \int_{-1}^{1} K(x+y) \hts{3,n}(y) \rmd y - g_{1,m,n}^N \Big) \rmd x \rmd t \right| \\
= & \left| \sum_{n=0}^{\infty} \sum_{\substack{1 \leq m \leq 2^N \\ m+n \text{ odd}}}  \int_{n \dt}^{(n+1)\dt} \int_{(m-1)\dx}^{(m+1)\dx} \varphi_1(x,t) \Big( \int_{-1}^{1} \big( K(x+y) - K(m\dx + y) \big) \hts{3,n}(y) \rmd y \Big) \rmd x \rmd t \right| \\
\leq & \| \varphi_1 \|_{L^1_{t,x}} \|\ts_3^N\|_{L^\infty_{t,x}} \max\{ \| K(\cdot) - K(2 [\frac{\cdot}{2 \dx}] \dx ) \|_{L^1}, \|  K(\cdot) - K(2 [\frac{\cdot+\dx}{2 \dx}] \dx - \dx) \|_{L^1} \} \\
\to & 0.
\end{align*}
And by the continuity of $\varphi_1$,
\begin{align*}
& \sum_{n=0}^{\infty} \sum_{\substack{1 \leq m \leq 2^N \\ m+n \text{ odd}}} \int_{n \dt}^{(n+1) \dt} \int_{(m-1)\dx}^{(m+1)\dx} (\varphi_1(x,t) - \varphi_1(x,n\dt)) g_{1,m,n}^N \rmd x \rmd t\\
\leq & \| \varphi_1(x,t) - \varphi_1(x,[\frac{t}{\dt}]\dt) \|_{L^1_t L^\infty_x } \max_{n} \sum_m |g_{1,m,n}^N| 
\overset{N \to \infty}{\longrightarrow} 0.
\end{align*}
Thus, one has
\begin{align*}
& \int_0^T \int_{0}^{1} \varphi_1(x,t) \Big( \int_{-1}^{1} K(x+y) \ts_3^N(y,t) \rmd y \Big) \rmd x \rmd t\\
& -\sum_{n=0}^{\infty} \sum_{\substack{1 \leq m \leq 2^N \\ m+n \text{ odd}}}^{} \int_{(m-1)\dx}^{(m+1)\dx} \varphi_1(x, n \dt) \dt g_{1,m,n}^N \rmd x \\
= & \int_{0}^{T} \int_{0}^{1} \varphi_1(x,t) \Big( \int_{-1}^{1} K(x+y) \big( \ts_3^N(y,t) - \ts_3^N(y, [\frac{t}{\dt}] \dt-) \big) \rmd y \Big) \rmd x \rmd t \\
& + \int_{0}^{T} \int_{0}^{1} \varphi_1(x,t) \Big( \int_{-1}^{1} K(x+y) \big( \ts_3^N (y,[\frac{t}{\dt}] \dt-) - \hts{3,[\frac{t}{\dt}]} (y) \big) \rmd y \Big) \rmd x \rmd t \\
& + \sum_{n=0}^{\infty} \sum_{\substack{1 \leq m \leq 2^N \\ m+n \text{ odd}}} \int_{n \dt}^{(n+1) \dt} \int_{(m-1)\dx}^{(m+1)\dx} \varphi_1(x,t) \Big( \int_{-1}^{1} K(x+y) \hts{3,n}(y) \rmd y - g_{1,m,n}^N \Big)  \rmd x \rmd t \\
& + \sum_{n=0}^{\infty} \sum_{\substack{1 \leq m \leq 2^N \\ m+n \text{ odd}}} \int_{n \dt}^{(n+1) \dt} \int_{(m-1)\dx}^{(m+1)\dx} (\varphi_1(x,t) - \varphi_1(x,n\dt)) g_{1,m,n}^N \rmd x \rmd t\\
\to & 0, \qquad \text{a.s. for $\vth$;\;  as }  N \to \infty.
\end{align*}
Similar result holds for $\ts_3$. 
Thus, $(\ts_1,\ts_3)^T$ satisfies \eqref{1.1_sys} in the sense of distribution.

Next, for each smooth convex entropy function $\eta = \eta(\sigma_1,\sigma_3)$ and the corresponding entropy flux function $\psi = \psi(\sigma_1,\sigma_3)$, 
one may multiply $\varphi \partial_{\sigma_1} \eta$ and $\varphi \partial_{\sigma_3} \eta$ to  \eqref{2.10_Burger} respectively, and integrate by parts to get
\begin{align*}
0 \geq & - \int_{0}^{T} \int_{-1}^{1} ( \eta \pt \varphi + \psi \px \varphi ) \rmd x \rmd t \\
& + \sum_{n=0}^{\infty} \int_{-1}^{1} \big( - \varphi \eta(x, n\dt+) + \varphi \eta(x,n \dt-) \big) \rmd x \\
& - \int_{-1}^{1} \varphi(x,0) \eta(\sigma_{1,0}(x), \sigma_{3,0}(x)) \rmd x.			
\end{align*}
By a similar argument as above, one can get that $(\ts_1,\ts_3)^T$ satisfies the entropy condition \eqref{1.8_entropy_eq} and thus is an entropy solution.

The periodicity of $(\ts_1,\ts_3)^T$ follows directly from that of $(\ts_1^N,\ts_3^N)^T$.
Then by choosing a sequence of test functions $\varphi_1^k \to \chi_{[0,1)\times [0,t]}$ and taking $\varphi_3 = 0$,
by the zero mean property of $K(x)$ \eqref{2.5_K}, one can get
\begin{equation*}
0 = \int_{0}^{1} \ts_1(x,t) \rmd x - \int_{0}^{t} \int_{0}^{1} \int_{-1}^{1} K(x+y) \ts_3(y,t) \rmd y \rmd x \rmd t = \int_{0}^{1} \ts_1(x,t) \rmd x.
\end{equation*}
Similarly, the zero mean property holds for $\ts_3$.
\end{proof}

According to Proposition \ref{prop:cons}, at each $t$, as long as \eqref{2.11_asmp_linf}--\eqref{2.13_asmp_l1dis} hold, for almost all choice of $\vartheta$, one can always choose a subsequence, still denoted as $(\ts_1^N,\ts_3^N)^T$, satisfying
\begin{equation} \label{2.17_mildmean}
\left| \int_{0}^{1} \ts_1^N(x,t) \rmd x \right| + \left| \int_{0}^{1} \ts_3^N(x,t) \rmd x \right|  \leq 1.
\end{equation}

\section{Growth Estimate}

In this section a rough estimate on the growth rate of the approximate solutions is given as follows
\begin{prop}\label{prop:growth}
If there is a time $T_0$ (for simplicity, assume $T_0 = n_0 \Delta t^{N_*}+$ for some $N^* \in \mbn$) and a subsequence of approximate solutions $(\ts_1^N,\ts_3^N)^T$ satisfying
\begin{gather}
\mz = \sup_N \{ \TV \ts_1^N(\cdot,T_0) + \TV \ts_3^N(\cdot,T_0) \} < +\infty, \label{2.14_t0bdd}\\
\left| \int_{0}^{1} \ts_1^N(x,T_0) \rmd x \right| + \left| \int_{0}^{1} \ts_3^N(x,T_0) \rmd x \right|  \leq 1, \label{2.15_balanced}
\end{gather}
then for any $T_*$ satisfying, for C.F.L. condition, 
\begin{equation}\label{2.16_cfl_tv}
\alpha (1+\frac{1}{2} \mz) \exp \left( \beta \me T_* \right) < {\Lambda},
\end{equation}
there exists a further subsequence such that \eqref{2.12_asmp_tv} holds in the sense
\begin{equation} \label{2.19_gr_tv}
\ms{t}  \leq \mz \exp \left( {\beta} \me (t - T_0) \right), \quad \forall\, t \in [T_0, T_0+T_*].
\end{equation}
Moreover, it holds
\begin{align}
\TV \ts_1^N(\cdot,t) \leq & \TV \ts_1^N (\cdot, T_0) + {\beta} \me (t-T_0) \max_{\tau \in [T_0,t]} \TV \ts_3^N(\cdot,\tau), \quad \forall\, t \in [T_0, T_0+T_*], \label{2.20_gr_tv1}\\
\TV \ts_3^N(\cdot,t) \leq & \TV \ts_3^N (\cdot, T_0) + {\beta} \me (t-T_0) \max_{\tau \in [T_0,t]} \TV \ts_1^N(\cdot,\tau), \quad \forall\, t \in [T_0, T_0+T_*]. \label{2.21_gr_tv3}
\end{align}
Meanwhile, for $t \in [T_0,T_0+T_*]$, \eqref{2.11_asmp_linf} and \eqref{2.13_asmp_l1dis} hold for
\begin{gather}
C_1 = 1 + \frac{\mz}{2} \exp({\beta} \me (t-T_0)), \label{2.22_gr_linf}\\
C_3 = \Big( 4 \Lambda \mz + {\beta} \me \mz  \Big) \exp({\beta} \me (t-T_0)) + {\beta} \me. \label{2.23_gr_L1dis}
\end{gather}
\end{prop}
\begin{proof}
In each time span $t \in ((n-1) \dt, n \dt) \subseteq [T_0,T_0+T_*]$ ($n \in \mathbb{Z}^+$), 
$(\ts_1^N,\ts_3^N)^T$ is constructed by solving finitely many Riemann problems \eqref{2.10_Burger} of two decoupled Burgers equations with piecewise constant initial data at $t = (n-1) \dt+$ on each period,
thus the total variation remains constant, i.e., for $i=1,3$
\[
\TV \ts_i^N(\cdot,n \dt-)  =
\TV \ts_i^N(\cdot,(n-1) \dt+) .
\]
While at the time $t = n \dt$, during the random sampling \eqref{2.7_hts}, the total variation can decrease only,
\[
\TV \hts{i,n} \leq \TV \ts_i^N(\cdot,n \dt-) = \TV \ts_i^N(\cdot,(n-1) \dt+).
\]
Therefore the only possibility that may increase the total variation is the effect of the nonlocal inhomogeneous terms.
Indeed, by \eqref{2.8_g}, \eqref{2.8_kmn_pro} and \eqref{2.9_tso},
\begin{align}
& \TV \ts_1^N(\cdot, n\dt+) \notag \\
= & \sum_{\substack{1 \leq m \leq 2^N \\ m+n \text{ odd}}}^{}  |\ts_{1,m,n}^N - \ts_{1,m-2,n}^N| \notag \\
\leq & \sum_{\substack{1 \leq m \leq 2^N \\ m+n \text{ odd}}}^{} \Big( |\hts{1,m,n} - \hts{1,m-2,n}| + |g_{1,m,n}^N - g_{1,m-2,n}^N| \dt \Big) \notag \\
= & \sum_{\substack{1 \leq m \leq 2^N \\ m+n \text{ odd}}} \Big( |\hts{1,m,n} - \hts{1,m-2,n}| + \Big| \sum_{\substack{-2^N < \tilde{m} \leq 2^N \\ \tilde{m} + n \text{ odd}}} (K_{m+\tilde{m}}^N - K_{m-2+\tilde{m}}^N ) \hts{3,\tilde{m},n} \Big| \dt \Big) \notag \\
= & \sum_{\substack{1 \leq m \leq 2^N \\ m+n \text{ odd}}} \Big( |\hts{1,m,n} - \hts{1,m-2,n}|  + \Big| \sum_{\substack{-2^N < \tilde{m} \leq 2^N \\ \tilde{m} + n \text{ odd}}} K_{m+\tilde{m}}^N ( \hts{3,\tilde{m},n} - \hts{3,\tilde{m}+2,n}) \Big| \dt \Big) \notag \\
\leq & \sum_{\substack{1 \leq m \leq 2^N \\ m+n \text{ odd}}}  |\hts{1,m,n} - \hts{1,m-2,n}|  + \Big| \sum_{\substack{-2^N < \tilde{m} \leq 2^N \\ \tilde{m} + n \text{ odd}}} \big( \sum_{\substack{1 \leq m \leq 2^N \\ m+n \text{ odd}}} | K_{m+\tilde{m}}^N| \big) | \hts{3,\tilde{m},n} - \hts{3,\tilde{m}+2,n}| \Big| \dt \notag \\
\leq & \TV \hts{1,n} + 2 \TV \hts{3,n}  \sum_m |K_{m}^N| \dt  \notag \\
\leq & \TV \ts_1^N(\cdot,(n-1)\dt+) + \TV \ts_3^N(\cdot,(n-1)\dt+) \cdot {\beta} \me \dt. \label{2.20_tv13}
\end{align}
Similarly, one has
\begin{equation}
\TV \ts_3^N(\cdot, n\dt+) \leq \TV \ts_3^N(\cdot,(n-1)\dt+) + \TV \ts_1^N(\cdot,(n-1)\dt+) \cdot {\beta} \me \dt. \label{2.21_tv31}
\end{equation}
Now, \eqref{2.20_gr_tv1} and \eqref{2.21_gr_tv3} follow directly.
Moreover, adding up \eqref{2.20_tv13}--\eqref{2.21_tv31} yields
\begin{align*}
& \TV \ts_1^N(\cdot, n\dt+) + \TV \ts_3^N(\cdot, n\dt+)  \\
\leq & \big( \TV \ts_1^N(\cdot,(n-1)\dt+) + \TV \ts_3^N(\cdot,(n-1)\dt+) \big) \cdot \big( 1+  {\beta} \me \dt \big). 
\end{align*}
Thus, for $N$ large,
\begin{align*}
& \TV \ts_1^N(\cdot,t) + \TV \ts_3^N(\cdot,t) \\
= & \TV \ts_1^N(\cdot, [\frac{t}{\dt}] \dt+) + \TV \ts_3^N (\cdot, [\frac{t}{\dt}] \dt+)\\
\leq & \Big( \TV \ts_1^N(\cdot, [\frac{T_0}{\dt}] \dt+) + \TV \ts_3^N (\cdot, [\frac{T_0}{\dt}] \dt+) \Big)  \big( 1 +  {\beta} \me \dt \big)^{[\frac{t-T_0}{\dt}]} \\
\leq & \mz \exp \left( {\beta} \me (t- T_0) \right),
\end{align*}
which proves \eqref{2.19_gr_tv}.

On the other hand, 
during the construction of the approximate solutions, 
the process of random sampling and solving Riemann problems for decoupled Burgers equations would not increase the $L^\infty$ norm, namely for $i=1,3$
\begin{equation*}
\max_m |\hts{i,m,n}|  \leq \| \ts_i^N(\cdot,n \dt-) \|_{L^\infty} = \| \ts_i^N(\cdot,(n-1) \dt+) \|_{L^\infty}.
\end{equation*}
Thus, only the inhomogeneous terms may increase the $L^\infty$ norm.
In fact, by the construction procedure \eqref{2.8_g} and \eqref{2.9_tso},
\begin{align*}
& \| \ts_1^N(\cdot,n \dt+) \|_{L^\infty} \\
= & \max_m |\ts_{1,m,n}^N| \\
\leq & \max_m |\hts{1,m,n}| + \max_m |g_{1,m,n}^N| \dt \\
= & \max_m |\hts{1,m,n}| + \max_m \left| \sum_{\substack{-2^N < \tilde{m} \leq 2^N \\ \tilde{m} + n \text{ odd}}} K_{m+\tilde{m}}^N \hts{3,\tilde{m},n}  \right| \dt \\
= & \max_m |\hts{1,m,n}| + \max_m |\hts{3,m,n}| \cdot \sum_m |K_m^N| \dt \\
\leq & \| \ts_1^N(\cdot, (n-1) \dt+) \|_{L^\infty} + \| \ts_3^N(\cdot, (n-1) \dt+) \|_{L^\infty} {\beta} \me \dt.
\end{align*}
Similar result holds for $\ts_3^N$. 
Thus, adding them up leads to
\begin{align*}
& \| \ts_1^N(\cdot, t) \|_{L^\infty} + \| \ts_3^N(\cdot, t) \|_{L^\infty} \\
\leq & \Big( \| \ts_1^N(\cdot, T_0) \|_{L^\infty} + \| \ts_3^N(\cdot, T_0) \|_{L^\infty} \Big) \exp({\beta} \me (t-T_0)) \\
\leq & \big( 1 + \frac{1}{2} \mz \big) \exp({\beta} \me (t-T_0)).
\end{align*}

Next, for $n \dt < t_1 < t_2 < (n+1)\dt$,
since the Riemann solvers are constants along the straight characteristics,
\begin{align*}
& \| \ts_1^N(\cdot,t_1) - \ts_1^N(\cdot,t_2) \|_{L^1} \\
= & \sum_{\substack{1 \leq m \leq 2^N \\ m+n \text{ even}}} \int_{(m-1) \dx}^{(m+1) \dx} | \ts_1^N(\cdot,t_1) - \ts_1^N(\cdot,t_2) | \rmd x \\
\leq & \sum_{\substack{1 \leq m \leq 2^N \\ m+n \text{ even}}} 2 \dx \TVm \ts_1^N(\cdot,n\dt+) \\
= & 2 \Lambda \dt \TV \ts_1^N(\cdot,n \dt+),
\end{align*}
while on the line $t = n \dt$,
\begin{align*}
& \| \ts_1^N(\cdot,n\dt+) - \ts_1^N(\cdot,n\dt-) \|_{L^1} \\
\leq & \sum_{\substack{1 \leq m \leq 2^N \\ m+n \text{ odd}}} |\ts_{1,m,n}^N - \hts{1,m,n}| 2 \dx + \int_{0}^{1} | \hts{1,n}(x) - \ts_1^N(x,n\dt-) |\rmd x \\
= & \sum_{\substack{1 \leq m\leq 2^N \\ m+n \text{ odd}}}  2 |g_{1,m,n}^N| \dt \dx \\
& \quad + \sum_{\substack{1 \leq m \leq 2^N \\ m+n \text{ odd}}}^{} \int_{(m-1)\dx}^{(m+1)\dx} | \ts_1^N( (m+\vth_n) \dx, n \dt- ) - \ts_1^N(x,n\dt-) | \rmd x \\
\leq & \max_m  \bigg| \sum_{\substack{ -2^N < \tilde{m} \leq 2^N \\ \tilde{m} + n \text{ odd} }} K_{\tilde{m} + m}^N \hts{3,\tilde{m},n}  \bigg| \dt + \sum_{\substack{1 \leq m \leq 2^N \\ m+n \text{ odd}}} \TVm \ts_1^N(\cdot,n \dt-) 2 \dx \\
\leq & {\beta} \me \max_m |\hts{3,m,n}| \dt + 2 \TV \ts_1^N(\cdot,n\dt-) \dx \\
\leq & {\beta} \me \| \ts_3^N(\cdot,n \dt-) \|_{L^\infty} \dt + 2 \Lambda \TV \ts_1^N(\cdot, n\dt-) \dt.
\end{align*}
Thus, for $t_1 < t_2$, one can combine the above inequalities, and similar ones for $\ts_3^N$, to get \eqref{2.13_asmp_l1dis} with 
\begin{equation} \label{3.11}
C_3 = \bigg( 4 \Lambda \mz + {\beta} \me (1 + \frac{1}{2} \mz) \bigg) \exp({\beta} \me T_*).
\end{equation}

It follows from the above estimates, as mentioned after Proposition \ref{prop:cons}, that \eqref{2.17_mildmean} holds for $t \in [T_0,T_0+T_*]$ for a subsequence of the approximate solutions.
With this and \eqref{2.19_gr_tv}, one can improve the estimates of $L^\infty$ to obtain \eqref{2.11_asmp_linf} with $C_1$ given in \eqref{2.22_gr_linf}, and then improve the above estimates \eqref{3.11} to get \eqref{2.13_asmp_l1dis} with $C_3$ given in \eqref{2.23_gr_L1dis}.
\end{proof}

It may be pointed out here that by this proposition, one can already get the global entropy solution by modifying the scheme with enlarging $\Lambda$ with time to avoid the violation of the C.F.L. condition.
Instead of doing this, we would like to give some much more detailed analysis in the next section to get the uniform a priori bound for the solutions,
which would not only show the global existence but also describe the behavior of the solutions.
Besides, a uniformly bounded solution is much more meaningful for the system \eqref{1.1_sys}, since it is an approximate system obtained through weakly nonlinear geometric optics approximation.

\section{Decay analysis}

In this section, it is shown that when the wave strengthes of $\ts_1$ and $\ts_3$ are much stronger than that of  $\ts_2$, 
the cancellation effect caused by genuine nonlinearity would dominate the effect of nonlinear resonance and make the solution decay.
In order to accomplish this analysis, the methods of approximate characteristics and approximate conservation laws originally developed in \cite{Glimm_Lax_1970} for the isentropic Euler system would be adopted to the system \eqref{1.1_sys}.

The result of this part could be summarized as follows:
\begin{prop}\label{prop:decay}
If there exist a time $T_0 = n_0 \Delta t^{N_*}+ $ and a subsequence of approximate solutions $(\ts_1^N,\ts_3^N)^T$ such that \eqref{2.17_mildmean} holds at $T_0$ and 
\begin{equation}
\msn{T_0} \in \big[ \frac{239}{300} \mb, \frac{4}{5} \mb \big],
\footnote{Here, $\mb$ is defined in \eqref{1.10_mb}, and the constants $\frac{239}{300}$ and $\frac{4}{5}$ are chosen accordingly to simplify the naration in the following proof. }
\end{equation}
then for 
\begin{equation}
T_* = \frac{60}{ \alpha \mb},
\end{equation}
there exists a further subsequence, such that
\begin{equation}
\ms{T_0+T_*} < \frac{239}{300} \mb
\end{equation}
and 
\begin{equation} \label{4.4_tv_bdd}
\ms{t} < \mb, \quad \forall\, t \in [T_0,T_0+T_*].
\end{equation}
\end{prop}

\begin{rem}
Combining Propositions \ref{prop:cons}, \ref{prop:growth} and \ref{prop:decay}, one can get easily the existence result of Theorem \ref{thm:main}.
\end{rem}

\begin{proof} 
First, according to \eqref{2.19_gr_tv} of Proposition \ref{prop:growth} and noting the definition of $\mb$ \eqref{1.10_mb}, it holds
\[
\ms{t} \leq \mz \exp(\frac{1}{5}) < \mb, \quad \forall\, t \in [T_0,T_0+T_*],
\]
which proves \eqref{4.4_tv_bdd}.
Moreover, it is easy to see that the system has a symmetry with respect to $\ts_1$ and $\ts_3$, thus without loss of generality, one may assume 
\[
\TV \ts_1^N(T_0) \geq \frac12 \ms{T_0} \geq  \TV \ts_3^N(T_0)
\]
at least for a subsequence.
Then by \eqref{2.21_gr_tv3} of Proposition \ref{prop:growth}, it holds
\begin{equation} \label{4.5_tv3_done}
\TV \ts_3^N(\cdot,t) < \frac{1}{2} \mz + {\beta} \me T_* \mb \leq \frac{3}{5} \mb, \quad \forall\, t \in [T_0,T_0+T_*].
\end{equation}
Thus, it remains to provide a bound for $\TV \ts_1^N(\cdot,T_0+T_*)$, which needs a detailed analysis on the decay of the solution.

The rest part of the proof is divided into 3 steps: first the approximate characteristics and approximate conservation laws are introduced and the uniform bounds for some useful quantities are given, then a suitable subsequence is chosen to pass to the limit, at last the decay is established through analyzing the widening effects of the rarefaction waves.

\subsection{Approximate characteristics and approximate conservation laws}

One may define the approximate characteristics as follows.
An approximate 1-characteristic is a union of line segments constructed according to the approximate solution $\ts_1^N$,
in which, each line segment is either a classical 1-characteristic or a 1-shock of the corresponding Burgers equation,
and its continuation starts from the diamond center that contains its ending,
meanwhile, for the choice of this continuation under different cases, one may follow the discussion given in Page 30 of \cite{Glimm_Lax_1970}.
Roughly speaking, the choice is made to prevent the 1-rarefaction wave from crossing the approximate 1-characteristic.
See Appendix C for the details.
In a similar manner, one may define the approximate 3-characteristics.

For each mesh diamond $\dm_{m,n}^N$ centering at $(m\dx,n\dt)$ ($(m+n)$ is even) in the construction of $(\ts_1^N,\ts_3^N)^T$,
we denote $\alpha_{i,m,n}^N, \beta_{i,m,n}^N$ as the $i$-waves entering $\dm_{m,n}^N$ from the southwest and southeast mesh edges respectively,
and $\gamma_{i,m,n}^N$ as the ones leaving from north edges, see Figure~\ref{fig:2_dm}.
Here, we also use them to denote the signed wave strength of the corresponding waves, such as 
\[
\alpha_{1,m,n}^N = \ts_1^N(m\dx,n\dt-) - \ts_1^N((m-1+\vth_n)\dx, n\dt-).
\]

\begin{rem}[To Be Deleted Before Submission]
Precisely,
\begin{align*}
\alpha_{1,m,n}^N = & \ts_1^N(m\dx,n\dt-) - \ts_1^N((m-1+\vth_n)\dx, n\dt-), \\
\beta_{1,m,n}^N = & \ts_1^N((m+1+\vth_n)\dx,n\dt-) - \ts_1^N(m\dx, n\dt-), \\
\gamma_{1,m,n}^N = & \ts_1^N((m+1+\vth_n)\dx,n\dt+) - \ts_1^N((m-1+\vth_n)\dx, n\dt+) \\
= & \ts_{1,m+1,n}^N - \ts_{1,m-1,n}^N
\end{align*}
and 
\begin{align*}
\alpha_{3,m,n}^N = & \ts_3^N((m-1+\vth_n)\dx, n\dt-) - \ts_3^N(m\dx,n\dt-), \\
\beta_{3,m,n}^N = & \ts_3^N(m\dx, n\dt-) - \ts_3^N((m+1+\vth_n)\dx,n\dt-), \\
\gamma_{3,m,n}^N = & \ts_{3,m-1,n}^N - \ts_{3,m+1,n}^N.
\end{align*}

\end{rem}

\emptline

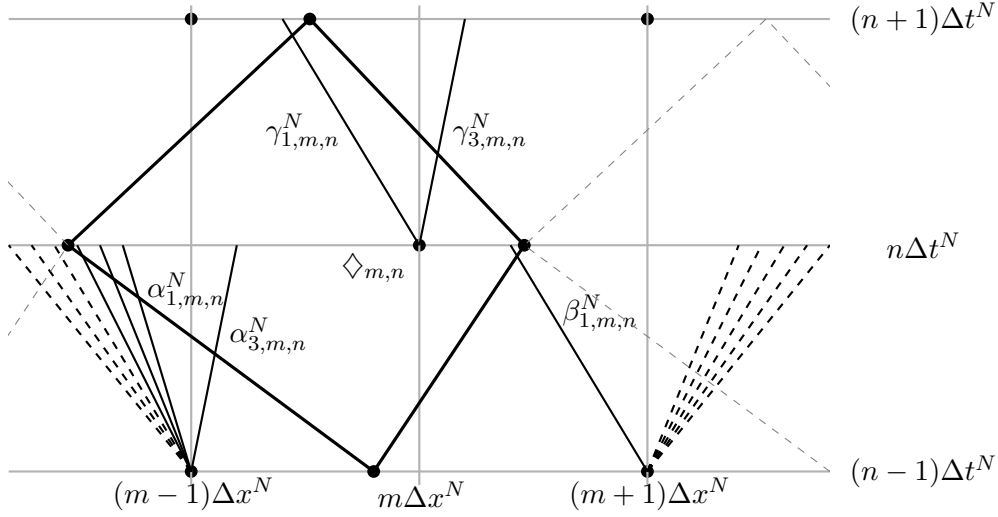
\begin{figure}[htbp]
\centering
\begin{tikzpicture}[scale = .6]

    \begin{scope}
    \clip (-9,-5.5) rectangle (9,5.5);
    \draw [black!30, thick] (-9,-5) -- (9,-5);
    \draw [black!30, thick] (-9,0) -- (9,0);
    \draw [black!30, thick] (-9,5) -- (9,5);
    \fill (5,5) circle (4pt);
    \fill (-5,5) circle (4pt);
    \fill (0,0) circle (4pt);
    \fill (5,-5) circle (4pt);
    \fill (-5,-5) circle (4pt);
    \draw [black!30, thick] (-5,-5.3) -- (-5,5.3);
    \draw [black!30, thick] (0,-5.3) -- (0,5.3);
    \draw [black!30, thick] (5,-5.3) -- (5,5.3);
    \draw [black, thick] (0,0) -- (-3,5) node[pos=.5,left] {$\gamma_{1,m,n}^N$}; 
    \draw [black, thick] (0,0) -- (1,5) node[pos=.5,right] {$\gamma_{3,m,n}^N$};
    \draw [black, thick, dashed] (-5,-5) -- (-9,0);
    \draw [black, thick, dashed] (-5,-5) -- (-8.5,0);
    \draw [black, thick, dashed] (-5,-5) -- (-8,0);
    \draw [black, thick] (-5,-5) -- (-7,0);
    \draw [black, thick] (-5,-5) -- (-7.5,0);
    \draw [black, thick] (-5,-5) -- (-6.5,0) node[pos=.8,right] {$\alpha_{1,m,n}^N$};

    \draw [black, thick] (-5,-5) -- (-4,0) node[pos=.6,right] {$\alpha_{3,m,n}^N$};
    
    \draw [black, thick] (5,-5) -- (2,0) node[pos=.7,right] {$\beta_{1,m,n}^N$};
    \draw [black, thick, dashed] (5,-5) -- (9,0);
    \draw [black, thick, dashed] (5,-5) -- (8.5,0);
    \draw [black, thick, dashed] (5,-5) -- (8,0);        
    \draw [black, thick, dashed] (5,-5) -- (7.5,0) (5,-5) -- (7,0);
    \fill (-7.7,0) circle (4pt);
    \fill (2.3,0) circle (4pt);
    \fill (-2.4,5) circle (4pt);
    \fill (-1,-5) circle (4pt);
    \draw [black, very thick] (-7.7,0) -- (-2.4,5) -- (2.3,0) -- (-1,-5) -- cycle;
    \draw [black!50, dashed] (-12.4,5) -- (-7.7,0) -- (-11,-5) -- (-17.7,0) -- cycle (9,-5) -- (2.3,0) -- (7.6,5) -- (12.3,0) -- cycle;
    \end{scope}
    \node (n) at (11,0) {$n \dt$};
    \node (n1) at (11,-5) {$(n-1) \dt$};
    \node (n2) at (11,5) {$(n+1) \dt$};
    \node (m) at (0,-5.5) {$m \dx$};
    \node (m1) at (-5,-5.5) {$(m-1) \dx$};
    \node (m2) at (5,-5.5) {$(m+1) \dx$};
    \node (mn) at (-1,-0.5) {$\diamondsuit_{m,n}$};
    
\end{tikzpicture}
\caption{Waves in one diamond}
\label{fig:2_dm}
\end{figure}

Then in the aforementioned approximate scheme, by \eqref{2.9_tso},
\begin{align}
\gamma_{1,m,n}^N = & \ts_{1,m+1,n}^N - \ts_{1,m-1,n}^N \notag \\
= & \hts{1,m+1,n} - \hts{1,m-1,n} - (g_{1,m+1,n}^N - g_{1,m-1,n}^N) \dt  \notag \\
= & \alpha_{1,m,n}^N + \beta_{1,m,n}^N - (g_{1,m+1,n}^N - g_{1,m-1,n}^N) \dt, \label{3.4_gamma} \\
\gamma_{3,m,n}^N = & \ts_{3,m-1,n}^N - \ts_{3,m+1,n}^N \notag \\
= & \alpha_{3,m,n}^N + \beta_{3,m,n}^N - (g_{3,m-1,n}^N - g_{3,m+1,n}^N) \dt. \label{3.4+_gamma}
\end{align}
For $i=1,3$, denote
\begin{align}
\Delta_i (\dm_{m,n}^N) & = \Big( | g_{i,m+1,n}^N - g_{i,m-1,n}^N | \Big) \dt, \label{3.5_Del} \\
C_i(\dm_{m,n}^N) & = \frac{1}{2} \Big( |\alpha_{i,m,n}^N| + |\beta_{i,m,n}^N| - |\alpha_{i,m,n}^N + \beta_{i,m,n}^N| \Big) \label{3.6_C}
\end{align}
the interfamily wave influence and intrafamily wave cancellation, respectively.
Then for $i = 1,3$,
\begin{gather}
|\gamma_{i,m,n}^N - (\alpha_{i,m,n}^N + \beta_{i,m,n}^N) | = |g_{i,m+1,n}^N - g_{i,m-1,n}^N| \dt = \Delta_i(\dm_{m,n}^N), \\
|\gamma_{i,m,n}^N| - \big( |\alpha_{i,m,n}^N | + | \beta_{i,m,n}^N| \big) \leq \Delta_i(\dm_{m,n}^N) - 2 C_i(\dm_{m,n}^N)
\end{gather}
and by \eqref{2.8_g}, \eqref{2.8_kmn_pro},
\begin{align}
\sum_{\substack{1\leq m \leq 2^N \\ m+n \text{ even} }} \Delta_1(\dm_{m,n}^N) \leq & \sum_m |K_{{m}}^N| \cdot \TV \hts{3,n} \dt \notag \\
\leq & {\beta} \me \dt \cdot \TV \ts_3^N(\cdot,n\dt-), \label{4.6_1del1}\\
\sum_{\substack{1\leq m \leq 2^N \\ m+n \text{ even} }} \Delta_3(\dm_{m,n}^N) 
\leq & {\beta} \me \dt \cdot \TV \ts_1^N(\cdot,n\dt-). \label{4.6_3del3}
\end{align}
If one denotes further the entering ($E$) and leaving ($L$) rarefaction ($+$) and shock ($-$) $i$-waves ($i=1,3$) of $\dm_{m,n}^N$ respectively as
\begin{align*}
E_i^+(\dm_{m,n}^N) & = \max \{ \alpha_{i,m,n}^N,0 \} + \max \{ \beta_{i,m,n}^N,0 \}, \\
E_i^-(\dm_{m,n}^N) & = \min \{ \alpha_{i,m,n}^N,0 \} + \min \{ \beta_{i,m,n}^N,0 \}, \\
L_i^+(\dm_{m,n}^N) & = \max \{ \gamma_{i,m,n}^N,0 \}, \\
L_i^-(\dm_{m,n}^N) & = \min \{ \gamma_{i,m,n}^N,0 \},
\end{align*}
then it holds approximate conservation laws for $\dm_{m,n}^N$ as
\[
L_i^\pm(\dm_{m,n}^N) = E_i^\pm(\dm_{m,n}^N) \mp C_i(\dm_{m,n}^N) + \delta \Delta_i(\dm_{m,n}^N).
\]
Here and hereafter, $\delta$ is a quantity taking values in $[-1,1]$, which may change its value from line to line.

Similar approximate conservation laws hold for other kinds of domains.
For instance, if $\Lambda^N$ is a domain composed by finitely many mesh diamonds, 
one can denote $E_i(\Lambda^N)$ and $L_i(\Lambda^N)$ as the waves entering $\Lambda^N$ from diamonds that not belonging to $\Lambda$ and the waves leaving $\Lambda^N$ to such diamonds respectively.
Then by adding up the above equations one can get
\[
L_i^\pm(\Lambda^N) = E_i^\pm (\Lambda^N) \mp C_i(\Lambda^N) + \delta \Delta_i(\Lambda^N),
\]
where 
\begin{align*}
C_i(\Lambda^N) \overset{\text{def.}}{=} & \sum_{\substack{m,n: \\
\dm_{m,n}^N \subseteq \Lambda^N}}  C_i(\dm_{m,n}^N), \\
\Delta_i(\Lambda^N) \overset{\text{def.}}{=} & \sum_{\substack{m,n: \\
\dm_{m,n}^N \subseteq \Lambda^N}}  \Delta_i(\dm_{m,n}^N).
\end{align*}

Especially, let $I^N$ be a horizontal interval connecting two mesh points on the line $t = n\dt$ and denote $\Lambda^N(I^N)$ as the union of the diamonds which contains the domain of determinacy of $I^N$ (See Figure~\ref{fig:2+_LamI}).
Since $L_i^+(\Lambda^N(I^N)) \geq 0$,
one has
\[
C_i(\Lambda^N(I^N)) \leq E_i^+(\Lambda^N(I^N)) + \Delta_i(\Lambda^N(I^N)).
\]
One may further denote $X_i^\pm(I^N)$ as the total signed strength of $i$-rarefaction waves/$i$-shocks passing through the horizontal interval $I^N$,
namely, entering from south into the mesh diamonds 
that contains $I^N$.
Due to the C.F.L. condition, all waves entering $\Lambda^N(I^N)$ from outside pass through $I^N$, thus,
\[
C_i(\Lambda^N(I^N)) \leq X_i^+(I^N) + \Delta_i(\Lambda^N(I^N)).
\]
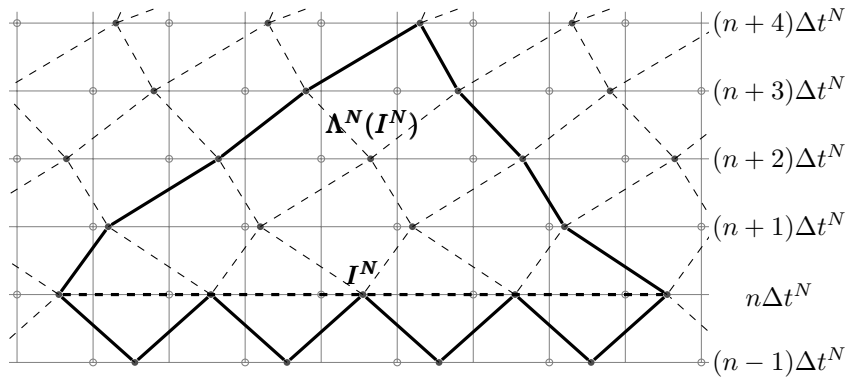
\begin{figure}[htbp]
\small
\centering
\begin{tikzpicture}[scale = .5,>=stealth]

\def\xspace{2}
\def\yspace{1.8}
\def\level{4}
\def\levelplus{5}
\def\waveno{5}
\def\wavenoplus{6}

\begin{scope}
\clip (1.8,-\yspace-0.2) rectangle (\waveno*\xspace*2+0.2,\level*\yspace+0.4);

\draw[help lines, xstep=\xspace ,ystep=\yspace ] (-1,-\yspace) grid (\waveno*\xspace*2+0.4,\level*\yspace+0.3); 

\foreach \x in {0, ..., \wavenoplus}
{
	\node[black!50,draw, inner sep =0.9pt, shape = circle] (O0\x) at  (\x*2*\xspace,-\yspace) {};
	\node[black!50,draw, inner sep =0.9pt, shape = circle] (O1\x) at  (\x*2*\xspace-\xspace,0) {};
	\node[black!50,draw, inner sep =0.9pt, shape = circle] (O2\x) at  (\x*2*\xspace,\yspace) {};
	\node[black!50,draw,inner sep =0.9pt, shape = circle] (O3\x) at  (\x*2*\xspace-\xspace,2*\yspace) {};
	\node[black!50,draw,inner sep =0.9pt, shape = circle] (O4\x) at  (\x*2*\xspace,3*\yspace) {};
	\node[black!50,draw,inner sep =0.9pt, shape = circle] (O5\x) at  (\x*2*\xspace-\xspace,4*\yspace) {};
}

\def\smpzero{-0.45} 
\def\smpone{0.55} 
\def\smptwo{0.2}
\def\smpthree{0.65}
\def\smpfour{-0.2}
\def\smpfive{-0.7}
\def\smpsix{-0.3}
\foreach \x in {0, ..., \wavenoplus}
{
	\node[black!70,fill,inner sep =0.9pt, shape = circle] (Oth0\x) at  (\x*2*\xspace-\xspace+\smpzero*\xspace,-\yspace) {};
	\node[black!70,fill,inner sep =0.9pt, shape = circle] (Oth1\x) at  (\x*2*\xspace-\xspace+\smpone*\xspace,0) {};
	\node[black!70,fill,inner sep =0.9pt, shape = circle] (Oth2\x) at  (\x*2*\xspace+\smptwo*\xspace,\yspace) {};
	\node[black!70,fill,inner sep =0.9pt, shape = circle] (Oth3\x) at  (\x*2*\xspace-\xspace+\smpthree*\xspace,2*\yspace) {};
	\node[black!70,fill,inner sep =0.9pt, shape = circle] (Oth4\x) at  (\x*2*\xspace-\xspace+\smpfour*\xspace,3*\yspace) {};
	\node[black!70,fill,inner sep =0.9pt, shape = circle] (Oth5\x) at  (\x*2*\xspace-\xspace+\smpfive*\xspace,4*\yspace) {};
	\node[black!70,fill,inner sep =0.9pt, shape = circle] (Oth6\x) at  (\x*2*\xspace-\xspace+\smpsix*\xspace,5*\yspace) {};
}

\foreach \x in {0, ..., \wavenoplus}
{
	\draw[dashed] (Oth0\x) -- (Oth1\x) ;
	\draw[dashed] (Oth0\x) ++(2*\xspace,0) -- (Oth1\x);
	\draw[dashed] (Oth1\x) -- (Oth2\x) ;
	\draw[dashed] (Oth1\x) ++(2*\xspace,0) -- (Oth2\x);
	\draw[dashed] (Oth2\x) -- (Oth3\x) ;
	\draw[dashed] (Oth2\x) ++(-2*\xspace,0) -- (Oth3\x);
	\draw[dashed] (Oth3\x) -- (Oth4\x) ;
	\draw[dashed] (Oth3\x) ++(-2*\xspace,0) -- (Oth4\x);
	\draw[dashed] (Oth4\x) -- (Oth5\x) ;
	\draw[dashed] (Oth4\x) ++(-2*\xspace,0) -- (Oth5\x);
	\draw[dashed] (Oth5\x) -- (Oth6\x) ;
	\draw[dashed] (Oth5\x) ++(-2*\xspace,0) -- (Oth6\x);
}
\end{scope}

\draw[black,very thick] (Oth54) -- (Oth43) -- node[right, pos=0.5] {$\qquad  \pmb{\Lambda^N(I^N)}$}  (Oth32) --  (Oth21) -- (Oth11) -- (Oth02) -- (Oth12) -- (Oth03) -- (Oth13) -- (Oth04) -- (Oth14) -- (Oth05) -- (Oth15);
\draw[black,very thick] (Oth54) -- (Oth44) -- (Oth34) -- (Oth24) -- (Oth15);
\draw[black,very thick,dashed] (Oth11) -- node[above,pos=0.5] {$\pmb{I^N}$}  (Oth15);

\node (n1) at (\wavenoplus*\xspace*2-\xspace,-\yspace) {$(n-1) \dt$};
\node (n1) at (\wavenoplus*\xspace*2-\xspace,\yspace*0) {$n \dt$};
\node (n2) at (\wavenoplus*\xspace*2-\xspace,\yspace*1) {$(n+1) \dt$};
\node (n3) at (\wavenoplus*\xspace*2-\xspace,\yspace*2) {$(n+2) \dt$};
\node (n4) at (\wavenoplus*\xspace*2-\xspace,\yspace*3) {$(n+3) \dt$};
\node (n5) at (\wavenoplus*\xspace*2-\xspace,\yspace*4) {$(n+4) \dt$};

\end{tikzpicture}
\caption{$\Lambda^N(I^N)$}
\label{fig:2+_LamI}
\end{figure}

Similarly, for a diamond $\dm_{m,n}^N$ that is cut through by an approximate $1$-characteristic $\chi^N$ into the left part $\dm_{m,n,L}^N$ and the right part $\dm_{m,n,R}^N$ (See Figure~\ref{fig:3_dmcut}),
one has the corresponding approximate conservation laws as 
\begin{equation} \label{4.*1}
L_1^+(\dm_{m,n,\LR}^N) = E_1^+(\dm_{m,n,\LR}^N) - C_1^+(\dm_{m,n,\LR}^N) + \delta \Delta_1(\dm_{m,n,\LR}^N),
\end{equation} 
where $\Delta_1(\dm_{m,n,\LR}) \geq 0$ denotes the amount of $3$-wave influence assigned to $\dm_{m,n,\LR}$, respectively,
and for the $1$-shocks leaving $\dm_{m,n,\LR}^N$ to join the inner boundary $\chi^N$, which is denoted as $S_1(\dm_{m,n,\LR}^N)$, it holds
\begin{equation} \label{4.*2}
S_1(\dm_{m,n,\LR}^N) = E_1^-(\dm_{m,n,\LR}^N) + C_1^-(\dm_{m,n,\LR}^N) + \delta \Delta_1(\dm_{m,n,\LR}).
\end{equation}
Here neither $E_1^-(\dm_{m,n,L}^N)$ nor $E_1^-(\dm_{m,n,R}^N)$ counts the $1$-shock entering $\dm_{m,n}^N$ along the inner boundary $\chi^N$ if any,
and $C_1^\pm(\dm_{m,n,\LR})$ denotes the $1$-waves canceled in the corresponding diamond halves,
for which one can get,
\begin{gather}
C_1^+(\dm_{m,n,L}^N) + C_1^+(\dm_{m,n,R}^N) = C_1(\dm_{m,n}^N), \label{4.**1}\\
C_1^-(\dm_{m,n,L}^N) + C_1^-(\dm_{m,n,R}^N) \leq C_1(\dm_{m,n}^N) \label{4.**2} \\
\Delta_1^+(\dm_{m,n,L}^N) + \Delta_1^+(\dm_{m,n,R}^N) = \Delta_1(\dm_{m,n}^N), \label{4.**3}
\end{gather}
and 
\[
C_1^+(\dm_{m,n,\LR}^N) \leq E_1^+(\dm_{m,n,\LR}^N) + \Delta_1(\dm_{m,n,\LR}^N).
\]
See Appendix C for the details.

Meanwhile, one can calculate the variation of $\ts_1^N$ on both sides of $\chi^N$ in $\dm_{m,n}^N$ as follows (see Figure \ref{fig:3_dmcut}): for $t_1 \in (n\dt,(n+1)\dt)$ and $t_2 \in ((n-1)\dt,n\dt)$, it holds that
\begin{align*}
& |\ts_1^N(\chi^N(t_1)+,t_1) - \ts_1^N(\chi^N(t_2)+,t_2) | \\
= & \Big| \big( \ts_1^N(\chi^N(t_1)+,t_1) - \ts_1^N( (m+1+\vth_n)\dx,n\dt+ ) \big) \\
& - \big( \ts_1^N(\chi^N(t_2)+,t_2) - \ts_1^N( (m+1+\vth_n)\dx,n\dt- ) \big) \Big| \\
= & | - L_1^+(\dm_{m,n,R}^N) + E_1^+(\dm_{m,n,R}^N) + E_1^-(\dm_{m,n,R}^N) | \\
\leq & E_1^+(\dm_{m,n,R}^N) + |E_1^-(\dm_{m,n,R}^N)| + 2 \Delta_1(\dm_{m,n,R}^N).
\end{align*}
Similarly,
\[
|\ts_1^N(\chi^N(t_1)-,t_1) - \ts_1^N(\chi^N(t_2)-,t_2) | \\
\leq E_1^+(\dm_{m,n,L}^N) + |E_1^-(\dm_{m,n,L}^N)| + 2 \Delta_1(\dm_{m,n,L}^N).
\]
\begin{figure}[htbp]
\small
\centering
\begin{tikzpicture}[scale = .5]
    \fill (0,0) circle (2pt) node[below] {$(m-1)\dx$};
    \fill (4,0) circle (2pt) node[below] {$m\dx$};
    \fill (8,0) circle (2pt) node[below] {$(m+1)\dx$} node[right] {$\quad 	(n-1) \dt$};
    \fill (8,3) circle (2pt) node[right] {$\quad \quad n\dt$};
    \fill (8,6) circle (2pt) node[right] {$\quad (n+1)\dt$};
    \fill (4,3) circle (2pt) (4,6) circle (2pt) (0,6) circle (2pt) (0,3) circle (2pt);

    \coordinate (S1) at (6,0);
    \coordinate (S2) at (0.8,3);
    \coordinate (S3) at (5.2,6);
    \coordinate (S4) at (8.8,3);
    \fill (S1) circle (2pt) (S2) circle (2pt) (S3) circle (2pt) (S4) circle (2pt);
    \draw (S1) -- (S2) -- (S3) -- (S4) -- cycle;

    \draw[black, very thick] (0,0) -- (3,3) node[pos=.5, left] {$\chi^N$}  (4,3) -- (7,6);
    \draw[black, very thick, dashed] (3,3) -- (4,3);
    
    \draw[black, dashed] (0,0) -- (2.4,3) (0,0) -- (2.6,3) (0,0) -- (2.8,3) (0,0) -- (3.3,3) (0,0) -- (3.6,3);
    \draw[black, dashed] (4,3) -- (7.3,6) (4,3) -- (6.4,6) (4,3) -- (6.6,6) (4,3) -- (6.8,6);
    \draw[black] (8,0) -- (8.6,3);

    \node (L) at (3,3.5) {$\diamondsuit_{m,n,L}^N$};
    \node (R) at (5,2) {$\diamondsuit_{m,n,R}^N$};
\end{tikzpicture}
~~~~
\begin{tikzpicture}[scale = .5]
    \fill (0,0) circle (2pt) node[below] {$(m-1)\dx$};
    \fill (4,0) circle (2pt) node[below] {$m\dx$};
    \fill (8,0) circle (2pt) node[below] {$(m+1)\dx$} node[right] {$\quad 	(n-1) \dt$};
    \fill (8,3) circle (2pt) node[right] {$\quad \quad n\dt$};
    \fill (8,6) circle (2pt) node[right] {$\quad (n+1)\dt$};
    \fill (4,3) circle (2pt) (4,6) circle (2pt) (0,6) circle (2pt) (0,3) circle (2pt);

    \coordinate (S1) at (4.4,0);
    \coordinate (S2) at (0.8,3);
    \coordinate (S3) at (3.2,6);
    \coordinate (S4) at (8.8,3);
    \fill (S1) circle (2pt) (S2) circle (2pt) (S3) circle (2pt) (S4) circle (2pt);
    \draw (S1) -- (S2) -- (S3) -- (S4) -- cycle;

    \draw[black, very thick] (8,0) -- (8.6,3) node[pos=.5, right] {$\chi^N$}  (4,3) -- (7.3,6);
    \draw[black, very thick, dashed] (8.6,3) -- (4,3);

    \draw[black, dashed] (0,0) -- (3,3) (0,0) -- (2.4,3) (0,0) -- (2.6,3) (0,0) -- (2.8,3) (0,0) -- (3.3,3) (0,0) -- (3.6,3);
    \draw[black, dashed] (4,3) -- (7,6) (4,3) -- (6.4,6) (4,3) -- (6.6,6) (4,3) -- (6.8,6);

    \node (L) at (3,3.5) {$\diamondsuit_{m,n,L}^N$};
    \node (R) at (5,2) {$\diamondsuit_{m,n,R}^N$};
\end{tikzpicture}
\caption{A diamond $\diamondsuit_{m,n}^N$ cut by an approximate 1-characteristic $\chi^N$}
\label{fig:3_dmcut}
\end{figure}
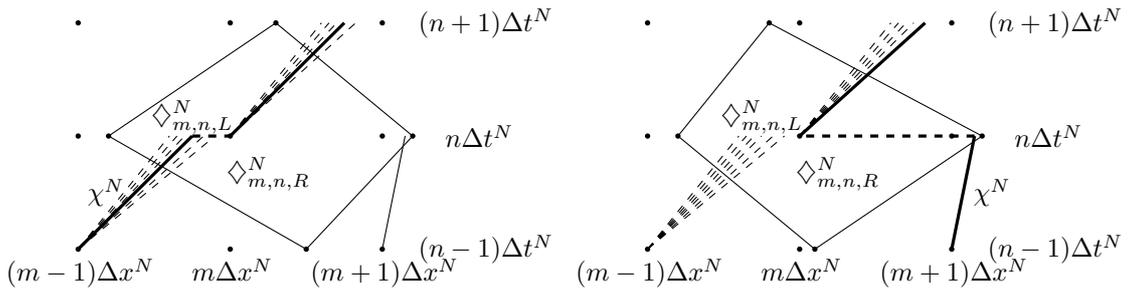

Adding up the above estimates, one can get the approximate conservation laws for the domain on one side of an approximate $1$-characteristic as follows.
For the part of an approximate $1$-characteristic $\chi^N$, 
that initiates from a diamond center $p_i^N$ of $\dm_{m_i,n_i}^N$, 
and finishes at another diamond center $p_f^N$ of $\dm_{m_f,n_f}^N$,
one can denote $\phi_R^N$ as the mesh curve composed of successive northeast edges of mesh diamonds from $\dm_{m_f,n_f}^N$ downwards to $\dm_{m_R,n_i}^N$  for some $m_R$, 
and denote $q_R^N$ as the ending points of $\phi_R^N$, 
namely the east mesh point of $\dm_{m_R,n_i}^N$,
denote $I_R^N$ as the horizontal interval connecting $p_i^N$ and $q_R^N$, with $\Lambda_R(\chi^N)$ as the domain surrounded by $\chi^N, \phi_R^N$ and south edges corresponding to $I_R^N$ (See Figure~\ref{fig:4_domain_cut}).
Then by adding up the approximate conservation laws of whole diamonds and diamond halves in $\Lambda_R(\chi^N)$, and due to the C.F.L. condition,
one can get
\begin{align*}
L_1^+(\Lambda_R(\chi^N)) = & X_1^+(I_R^N) - C_1^+(\Lambda_R(\chi^N)) + \delta \Delta_1(\Lambda_R(\chi^N)), \\
S_1(\Lambda_R(\chi^N)) + L_1^-(\Lambda_R(\chi^N)) = & X_1^-(I_R^N) + C_1^-(\Lambda_R(\chi^N)) + \delta \Delta_1(\Lambda_R(\chi^N)).
\end{align*}
Then obviously,
\begin{align*}
L_1^+(\Lambda_R(\chi^N)) \leq & X_1^+(I_R^N) + \Delta_1(\Lambda_R(\chi^N)),\\
|S_1(\Lambda_R(\chi^N))| \leq & |X_1^-(I_R^N)| + \Delta_1(\Lambda_R(\chi^N)).
\end{align*}
Similarly, one can construct the left side domain $\Lambda_L(\chi^N)$ and get the corresponding approximate conservation laws
\begin{align*}
L_1^+(\Lambda_L(\chi^N)) \leq & X_1^+(I_L^N) + \Delta_1(\Lambda_L(\chi^N)),\\
|S_1(\Lambda_L(\chi^N))| \leq & |X_1^-(I_L^N)| +\Delta_1(\Lambda_L(\chi^N)).
\end{align*}
Here $E_1^\pm(\Lambda_\LR(\chi^N)), L_1^\pm(\Lambda_\LR(\chi^N))$ denote the waves entering and leaving $\Lambda_\LR(\chi^N)$, $S_1(\Lambda_\LR(\chi^N))$ denotes the waves entering the boundary at the approximate characteristic $\chi^N$,
$C_1^\pm (\Lambda_\LR(\chi^N)), \Delta_1^\pm (\Lambda_\LR(\chi^N)) $ are the total amount of the corresponding values.
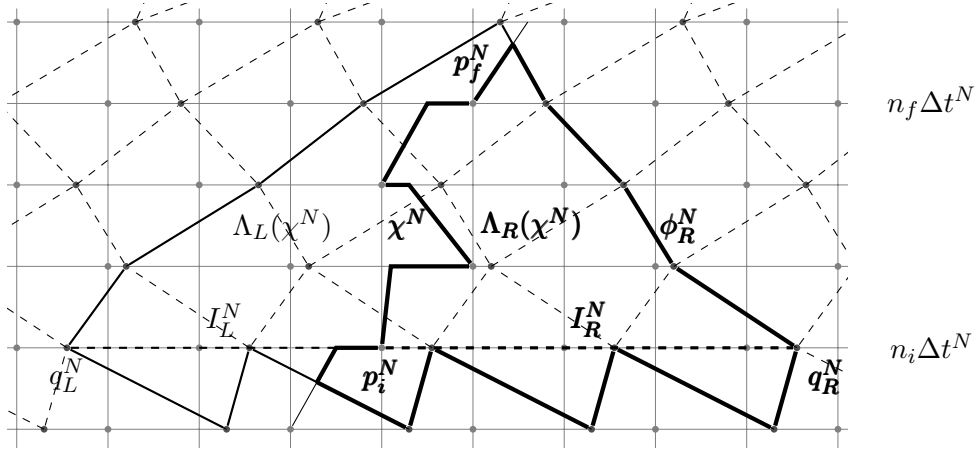
\begin{figure}[htbp]
\centering
\begin{tikzpicture}[scale = .6,>=stealth]

\def\xspace{2}
\def\yspace{1.8}
\def\level{4}
\def\levelplus{5}
\def\waveno{5}
\def\wavenoplus{6}

\begin{scope}
\clip (1.8,-2.2) rectangle (\waveno*\xspace*2+0.2,\level*\yspace+0.4);

\draw[help lines, xstep=\xspace ,ystep=\yspace ] (-1,-3) grid (\waveno*\xspace*2+0.4,\level*\yspace+0.3); 

\foreach \x in {0, ..., \wavenoplus}
{
	\node[black!50,fill,inner sep =0.9pt, shape = circle] (O0\x) at  (\x*2*\xspace,-\yspace) {};
	\node[black!50,fill,inner sep =0.9pt, shape = circle] (O1\x) at  (\x*2*\xspace-\xspace,0) {};
	\node[black!50,fill,inner sep =0.9pt, shape = circle] (O2\x) at  (\x*2*\xspace,\yspace) {};
	\node[black!50,fill,inner sep =0.9pt, shape = circle] (O3\x) at  (\x*2*\xspace-\xspace,2*\yspace) {};
	\node[black!50,fill,inner sep =0.9pt, shape = circle] (O4\x) at  (\x*2*\xspace,3*\yspace) {};
	\node[black!50,fill,inner sep =0.9pt, shape = circle] (O5\x) at  (\x*2*\xspace-\xspace,4*\yspace) {};
}

\def\smpzero{0.3}
\def\smpone{0.55} 
\def\smptwo{0.2}
\def\smpthree{0.65}
\def\smpfour{-0.2}
\def\smpfive{-0.7}
\def\smpsix{-0.3}
\foreach \x in {0, ..., \wavenoplus}
{
	\node[black!70,fill,inner sep =0.9pt, shape = circle] (Oth0\x) at  (\x*2*\xspace-\xspace+\smpzero*\xspace,-\yspace) {};
	\node[black!70,fill,inner sep =0.9pt, shape = circle] (Oth1\x) at  (\x*2*\xspace-\xspace+\smpone*\xspace,0) {};
	\node[black!70,fill,inner sep =0.9pt, shape = circle] (Oth2\x) at  (\x*2*\xspace+\smptwo*\xspace,\yspace) {};
	\node[black!70,fill,inner sep =0.9pt, shape = circle] (Oth3\x) at  (\x*2*\xspace-\xspace+\smpthree*\xspace,2*\yspace) {};
	\node[black!70,fill,inner sep =0.9pt, shape = circle] (Oth4\x) at  (\x*2*\xspace-\xspace+\smpfour*\xspace,3*\yspace) {};
	\node[black!70,fill,inner sep =0.9pt, shape = circle] (Oth5\x) at  (\x*2*\xspace-\xspace+\smpfive*\xspace,4*\yspace) {};
	\node[black!70,fill,inner sep =0.9pt, shape = circle] (Oth6\x) at  (\x*2*\xspace-\xspace+\smpsix*\xspace,5*\yspace) {};
}

\foreach \x in {0, ..., \wavenoplus}
{
	\draw[dashed] (Oth0\x) -- (Oth1\x) ;
	\draw[dashed] (Oth0\x) ++(2*\xspace,0) -- (Oth1\x);
	\draw[dashed] (Oth1\x) -- (Oth2\x) ;
	\draw[dashed] (Oth1\x) ++(2*\xspace,0) -- (Oth2\x);
	\draw[dashed] (Oth2\x) -- (Oth3\x) ;
	\draw[dashed] (Oth2\x) ++(-2*\xspace,0) -- (Oth3\x);
	\draw[dashed] (Oth3\x) -- (Oth4\x) ;
	\draw[dashed] (Oth3\x) ++(-2*\xspace,0) -- (Oth4\x);
	\draw[dashed] (Oth4\x) -- (Oth5\x) ;
	\draw[dashed] (Oth4\x) ++(-2*\xspace,0) -- (Oth5\x);
	\draw[dashed] (Oth5\x) -- (Oth6\x) ;
	\draw[dashed] (Oth5\x) ++(-2*\xspace,0) -- (Oth6\x);
}

\end{scope}


\coordinate (P1) at (6.6*\xspace,4*\yspace);

\path [name path = curve 1] (Oth54) -- (Oth44);
\path [name path = curve 2] (P1)-- (O43);

\path [name path = curve 3] (Oth03) -- (Oth12);
\path [name path = curve 4] (O13) -- ++(-0.5*\xspace,0) -- (O02);


	
\fill [name intersections = {of = curve 1 and curve 2, by={H}}] 
			(H) circle (1pt) ;
\fill [name intersections = {of = curve 3 and curve 4, by={T}}] 
			(T) circle (1pt) ;

\draw[black,thick] (H) -- (Oth54) -- (Oth43) --  (Oth32) -- node[right,pos=0.5]{$\quad {\Lambda_L(\chi^N)}$} (Oth21) -- (Oth11) node[below] {$q_L^N$} -- (Oth02) -- (Oth12) -- (T);
\draw[black,thick,dashed] (Oth11) -- node[above,pos=0.5] {$I_L^N$} (O13);
\draw[black,ultra thick] (H) -- (Oth44) -- (Oth34) -- node[right,pos=0.5] {$\pmb{\phi_R^N}$} (Oth24) -- (Oth15) node[below right] {$\pmb{q_{R}^N}$} -- (Oth05) -- (Oth14) -- (Oth04) -- (Oth13) -- (Oth03) -- (T);
\draw[black,ultra thick] (H) -- (O43) -- ++(-0.5*\xspace,0) -- (O33) -- ++(0.3*\xspace,0) -- node[left,pos=0.5] {$\pmb{\chi^N}$} node[right,pos=0.5] {  $\quad \pmb{\Lambda_R(\chi^N)}$} (O23) -- ++(-0.9*\xspace,0) -- (O13) node[below] {$\pmb{p_i^N}$} -- ++(-0.5*\xspace,0) -- (T) ;
\draw[black]  (T) -- (O02) (H) -- (P1);
\draw[black,very thick, dashed] (O13) -- node[above,pos=0.5] {$\pmb{I_R^N}$} (Oth15) ;

\node[label=above:{$\pmb{p_f^N}$}] (P) at (O43) {};

\node (n1) at (\wavenoplus*\xspace*2-\xspace,\yspace*0) {$n_i \dt$};
\node (n4) at (\wavenoplus*\xspace*2-\xspace,\yspace*3) {$n_f \dt$};

\end{tikzpicture}
\caption{$\Lambda_R(\chi^N)$ and $\Lambda_L(\chi^N)$}
\label{fig:4_domain_cut}
\end{figure}


Meanwhile, the total variation of $\ts_1^N$ on the right side of $\chi^N$ can be estimated as 
\[
\TVchi{+} \ts_1^N \leq X_1^+(I_R^N) + |X_1^-(I_R^N)| + 2 |\Delta_1(\Lambda_R(\chi^N))|,
\] 
and for the left side
\[
\TVchi{-} \ts_1^N \leq X_1^+(I_L^N) + |X_1^-(I_L^N)| + 2 |\Delta_1(\Lambda_L(\chi^N))|,
\] 
Then for the total variation of the speed of $\chi^N$, it holds
\begin{align*}
\TVchi{} \dot{\chi}^N \leq & \frac{1}{2} \Big( \TVchi{+} \ts_1^N +  \TVchi{-} \ts_1^N\Big) \\
\leq & \frac{1}{2} \Big( X_1^+(I_R^N \cup I_L^N) + |X_1^-(I_R^N \cup I_L^N)| \Big) +  |\Delta_1(\Lambda_R(\chi^N) \cup \Lambda_L(\chi^N))|.
\end{align*}

Since the domain under study is covered by the diamonds centering in $[T_0,T_0+T_*] \times [0,1]$ with $T_* \leq 60 / \mb$,
if one chooses $\kappa = [2 \Lambda T_*] + 2$, where $\Lambda$ is the one for C.F.L. condition, and
\[
I^* = \{ (x,T_0) \mid x \in [-\kappa,\kappa)  \}
\]
then each above domain locates in the determinacy domain of $I^*$, moreover,
\[
| X_i^\pm (I^*) | = \kappa \TV \ts_i^N(T_0),
\]
which is bounded.
Meanwhile, by \eqref{4.6_1del1}--\eqref{4.6_3del3}, and \eqref{4.4_tv_bdd}, \eqref{1.10_mb},
for each domain $\Omega$ locates in the determinacy domain of $I^*$, 
\[
\Delta_1(\Omega) + \Delta_3(\Omega) \leq \frac{\kappa}{5} \mb,
\]
which is bounded.
Using the approximate conservation laws given above, all listed quantities, such as $C_1(\Omega), S(\chi^N), \TVchi{} \ts_i^N, \TVchi{} \dot{\chi}^N$, are bounded.

\subsection{Estimates for the exact solution} \label{ssc:4.2}

With above estimates, 
one may pass to the limit to the exact solution.
And during this process, a subsequence of approximate solutions could be selected to satisfy the properties as follows.

\begin{lem} \label{lem:4.2}
For a sequence of approximate $1$-characteristics $\{\chi^N\}$, 
if there is a sequence of time $\{ t^N \}$ such that $(\chi^N(t^N),t^N)$ converges to some point $(x_0,t_0)$,
then it possess a convergent subsequence
\[
\chi^N(t) \to \chi(t)
\]
uniformly in time $t \in [T_0,T_0+T_*]$.
Moreover, the limit $\chi(t)$ is Lipschitz continuous.
\end{lem}


\begin{lem} \label{lem:4.3}
For $\chi$ as above,
if it holds 
\[
a \leq \dot{\chi}^N \leq b,
\]
for large enough $N$, then
\[
a \leq \dot{\chi} \leq b.
\]
Moreover, 
\[
\lim_{N} \dot{\chi}^N(t) = \dot{\chi}(t)
\]
at all but a countable set of $t$ for a further subsequence.
\end{lem}

\begin{lem} \label{lem:4.4}
For $\chi$ as above, there exists a further subsequence such that $(\ts_1^N,\ts_3^N)^T$ are one-sided equicontinuous on both sides of $\chi^N$ except for a countable set of $t$, and it holds that
\[
\lim_N \ts_1^N(\chi^N(t)\pm 0,t) = \ts_1(\chi(t)\pm 0,t).
\]
\end{lem}

\begin{lem} \label{lem:4.5}
Except for a countable set of $t$, it holds that either
\[
\dot{\chi}(t) = \frac{\alpha [\frac{1}{2}\ts_1^2]}{[\ts_1]}\Big|_{(\chi(t),t)} = \frac{\alpha}{2} \big( \ts_1(\chi(t)+,t) + \ts_1(\chi(t)-,t) \big)
\]
or 
\[
\dot{\chi}(t) = \alpha \ts_1(\chi(t),t).
\]
\end{lem}

The proof of the above lemmas is similar to the one given in \cite{Glimm_Lax_1970}.
See Appendix B for the details.

Denote $\rmd C_i^N$ and $\rmd \Delta_i^N$ as the measures corresponding to the approximate solution $(\ts_1^N,\ts_3^N)^T$ that assign its value in each diamond $\dm_{m,n}^N$ to the center $(m \dx, n \dt)$.
Due to the bounds of $C_i^N(\Omega)$ and $\Delta_i^N(\Omega)$, one has
\[
\rmd C_i^N \to \rmd C_i, \quad \rmd \Delta_i^N \to \rmd \Delta_i
\]
in weak*-topology for a subsequence.

Similarly, one may define the absolute value of the wave strength $\Str \chi(t)$ for the wave on one characteristic $\chi(t)$ and prove the corresponding convergence for the approximate sequence.
And one may call two characteristics of the same family $\chi_1$ and $\chi_2$ as coalescing, if there are infinitely many approximate ones coalesce.

\subsection{Widening effects of the rarefaction waves}

Let us focus on the approximate solutions on the domain between two approximate $1$-characteristics (See Figure~\ref{fig:5_char}).
For two $1$-characteristics $\chi_1(t) = \lim \chi_1^N(t)$ and $\chi_2(t) = \lim \chi_2^N(t)$ with $0 \leq \chi_2(t) - \chi_1(t) \leq 1$, 
denote
\[
I(t) = \overline{\chi_1(t) \chi_2(t)}, \quad D(t) = |I(t)| = \chi_2(t) - \chi_1(t).
\]
Then $\dot{D}(t) = \dot{\chi}_2(t) - \dot{\chi}_1(t)$ and
\[
D(t) = D(T_0) + \int_{T_0}^{t} \Big( \dot{\chi}_2(\tau) - \dot{\chi}_1(\tau) \Big) \rmd \tau.
\]

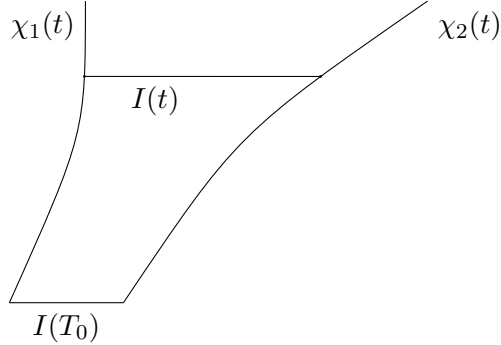
\begin{figure}[htbp]
\centering
\begin{tikzpicture}[scale = .5]
		
	\draw [name path = curve 1] (0,0) .. controls (2,4.5) and (2,4.5) .. (2,8) node[below left] {$\chi_1(t)$};
	\draw [name path = curve 2] (3,0) .. controls (6,4.5) and (6,4.5) .. (11,8) node[below right] {$\chi_2(t)$};
	\path [name path = curve 3] (0,6) -- (11,6);
	
	\fill [name intersections = {of = curve 1 and curve 3, by={a1}}] 
			(a1) circle (1pt) ;
	\fill [name intersections = {of = curve 2 and curve 3, by={a2}}] 
			(a2) circle (1pt);
		
	\draw[black] (a1)-- node[below,pos=0.3] {$I(t)$} (a2);
	\draw[black] (0,0) -- node[below,pos=0.5] {$I(T_0)$} (3,0);

\end{tikzpicture}
\caption{Domain between two characteristics}
\label{fig:5_char}
\end{figure}

For an approximate solution $(\ts_1^N,\ts_3^N)^T$, 
\begin{align*}
& \dot{\chi}_2^N(n_*^N\dt) - \dot{\chi}_1^N(n_*^N\dt) \\
= & \frac{\alpha}{2} \Big( \ts_1^N(\chi_2^N(n_*^N\dt)+, n_*^N \dt) + \ts_1^N(\chi_2^N(n_*^N \dt)-,n_*^N\dt) \Big) \\
& -  \frac{\alpha}{2} \Big( \ts_1^N(\chi_1^N(n_*^N\dt)+, n_*^N \dt) + \ts_1^N(\chi_1^N(n_*^N \dt)-,n_*^N\dt) \Big) \\
= &  \frac{\alpha}{2} \Big( 2 \ts_1^N(\chi_2^N(n_*^N\dt)-, n_*^N \dt) - \Str \chi_2^N(n_*^N \dt)  \Big) \\
& - \frac{\alpha}{2} \Big( 2 \ts_1^N(\chi_1^N(n_*^N\dt)+, n_*^N \dt) + \Str \chi_1^N(n_*^N \dt)  \Big) \\
= &  {\alpha} \Big( \ts_1^N(\chi_2^N(n_*^N\dt)-, n_*^N \dt) - \ts_1^N(\chi_1^N(n_*^N\dt)+, n_*^N \dt) \\
& - \frac{1}{2} \Str \chi_2^N(n_*^N \dt) - \frac{1}{2} \Str \chi_1^N(n_*^N \dt) \Big)\\
= & \alpha X_1^+(I^N(n_*^N\dt)) + \alpha X_1^-(I^N(n_*^N \dt)) - \frac{\alpha}{2} \big( \Str \chi_2^N(n_*^N \dt) + \Str \chi_1^N(n_*^N \dt) \big). 
\end{align*}
Passing to the limit leads to
\[
\dot{D}(t) = \dot{\chi}_2(t) - \dot{\chi}_1(t) = \alpha X_1^+(I(t)) + \alpha X_1^-(I(t)) - \frac{\alpha}{2} \big( \Str \chi_2(t) + \Str \chi_1(t) \big),
\]
then 
\[
D(t) = D(T_0) + \alpha \int_{T_0}^{t} \big( X_1^+(I(\tau)) + X_1^-(I(\tau)) \big) \rmd \tau - \frac{\alpha}{2} \int_{T_0}^{t} \big( \Str \chi_2(\tau) + \Str \chi_1(\tau) \big) \rmd \tau.
\]
On the other hand,
by the approximate conservation laws on the domain $\Lambda^N_{t_1,t_2}$ surrounded by $\chi_1^N, \chi_2^N, t=t_1, t = t_2 $ and passing to the limit, it holds
\begin{align*}
X_1^+(I(\tau)) \geq & X_1^+(I(t)) - \Delta_1(\Lambda_{\tau,t}),\\
|X_1^-(I(\tau))| \leq & |X_1^-(I(T_0))| + \Delta_1(\Lambda_{T_0,\tau}).
\end{align*}
Thus,
\[
D(t) \geq \alpha (t-T_0) \big( X_1^+(I(t)) - |X_1^-(I(T_0))| \big) - \alpha (t-T_0) \Delta_1(\Lambda_{T_0,t}) - \frac{\alpha}{2} \int_{T_0}^{t} \big( \Str \chi_2(\tau) + \Str \chi_1(\tau) \big) \rmd \tau. 
\]
and 
\[
X_1^+(I(t)) \leq \frac{D(t)}{\alpha (t-T_0)} + |X_1^-(I(T_0))| + \Delta_1(\Lambda_{T_0,t}) + \frac{1}{2} \frac{1}{t-T_0} \int_{T_0}^{t} \big( \Str \chi_2(\tau) + \Str \chi_1(\tau)  \big) \rmd \tau.
\]

Now one can get rid of $\Str \chi_1$ and $\Str \chi_2$ by the following procedure as in \cite{Glimm_Lax_1970}.
Divide $I(T_0)$ into small pieces with $\xi_l$ as the corresponding dividing points,
such that 
for smooth data case, at $T_0$ the $1$-rarefaction waves crossing each $\overline{\xi_l \xi_{l+1}}$ is not larger that $\frac{1}{800} |I(T_0)| \mb$.
Denote $\psi_l$ as a $1$-characteristic originating from $\xi_l$ and denote $\Lambda_{l,l+1}$ as the domain surrounded by $I(T_0), I(T_0+T_*), \psi_l$ and $\psi_{l+1}$.
Then one can find the first $\psi_{l^*}$ that does not coalesce with $\chi_1$ and the last $\psi_{l^{**}}$ that does not coalesce with $\chi_2$, where, without loss of generality, one may assume $\chi_2$ does not coalesce with $\chi_1$.
Repeat the above process to $\Lambda_{l^*,l^{**}}$, noting that $\Str \psi_{l^*}(t)$ and $\Str \psi_{l^{**}}(t)$ are parts of $X_1^-(I(t))$,
and by applying the approximate conservation laws to $\Lambda_{l^*-1,l^*}$ and $\Lambda_{l^{**},l^{**}+1}$, one has
\begin{equation} \label{4.13_wdn}
X_1^+(I(t)) \leq \frac{D(t)}{\alpha (t-T_0)} - X_1^-(I(T_0)) + \Delta_1(\Lambda_{T_0,t}) + \frac{1}{400} |I(T_0)| \mb.
\end{equation}
For the case that the data is not smooth at $T_0$, one may take an approximate sequence.

At last, taking a sequence of approximate smooth data if necessary, 
one can divide the data at $t=T_0$ into pieces with the starting points $\zeta_j= (\chi_j(T_0),T_0)$ of $1$-characteristic curves $\chi_j$, such that
\begin{align*}
& X_1^+(\overline{\zeta_j \zeta_{j+1}}) \leq \delta_j, \quad j \text{ is odd,} \\
& \left| X_1^- (\overline{\zeta_j \zeta_{j+1}}) \right| \leq \delta_j, \quad j \text{ is even}
\end{align*}
with 
\[
\sum_j \delta_j \leq \frac{1}{400} \mb.
\]
Then, for $\zeta_j^* = (\chi_j(T_0+T_*),T_0+T_*)$, simply by approximate conservation laws on the domain $\Lambda_{j,j+1}$ surrounded by $\chi_j,\chi_{j+1}, t=T_0$ and $t=T_0+T_*$, it holds
\begin{align*}
& X_1^+( \overline{\zeta_j^* \zeta_{j+1}^*}) \leq \delta_j + \Delta_1(\Lambda_{j,j+1}), \quad \text{for } j \text{ odd,}\\
\intertext{while by \eqref{4.13_wdn},}
& X_1^+( \overline{\zeta_j^* \zeta_{j+1}^*}) \leq \frac{|  \overline{\zeta_j^* \zeta_{j+1}^*} | }{\alpha T_*} + \delta_j + \Delta_1(\Lambda_{j,j+1}) + \frac{1}{400} |\overline{\zeta_j \zeta_{j+1}}| \mb, \quad \text{for } j \text{ even.}
\end{align*}
Noting that
\begin{equation*}
\sum_{j} |  \overline{\zeta_j^* \zeta_{j+1}^*} | = 1, \quad \sum_j |\overline{\zeta_j\zeta_{j+1}}| = 1, 
\end{equation*}
adding these estimates, one has
\[
X_1^+( [0,1) \times \{T_0+T_*\} ) \leq \frac{1}{\alpha T_*} + \frac{1}{200} \mb + \sum_{j} \Delta_1(\Lambda_{j,j+1}).
\]
Meanwhile, due to the periodicity, by \eqref{4.6_1del1}, \eqref{4.5_tv3_done} and \eqref{1.10_mb},
\[
\sum_{j} \Delta_1(\Lambda_{j,j+1}) \leq  \Delta_1([0,1) \times [T_0,T_0+T_*]) \leq \frac{3}{50} \mb.
\]
Thus, one has
\begin{equation*}
X_1^+( [0,1) \times \{T_0+T_*\} ) \leq \frac{49}{600} \mb.
\end{equation*}
By periodicity,
\[
\TV \ts_1(T_0+T_*) \leq \frac{49}{300} \mb.
\]
Combining this with \eqref{4.5_tv3_done}, one can conclude the proof.
\end{proof}

At last, the proof for the uniqueness is remarked briefly here.
Since the quasilinear part of the system is two decoupled Burgers equations,
while the interaction terms can be treated linearly, 
one may perform the method of S.N. Kruzhkov \cite{Kruzkov_1970} to show the uniqueness as follows.
Suppose $(\ts_1^*,\ts_3^*)^T$ and $(\ts_1^{**},\ts_3^{**})^T$ are both entropy solutions to the Cauchy problem \eqref{1.1_sys} on $(x,t) \in [0,L] \times [0,T]$ with same periodic initial data $(\ts_{1,0},\ts_{3,0})^T$ and satisfy
\begin{gather*}
	\ts_1(x+1,t) = \ts_1(x,t), \quad \ts_3(x+1,t) = \ts_3(x,t), \quad \forall\, (x,t) \in [0,L] \times [0,T], \\
	\int_0^1 \ts_1(x,t) \rmd x = 0, \quad \int_0^1 \ts_3(x,t) \rmd x = 0, \quad \forall\, t \in [0,T],\\
	\| \ts_1 \|_{L_{x,t}^\infty} + \| \ts_3 \|_{L_{x,t}^\infty} \leq C_1
\end{gather*}
for some $C_1>0$ and for $(\ts_1,\ts_3)^T = (\ts_1^*,\ts_3^*)^T$ and $(\ts_1^{**},\ts_3^{**})^T$. 
Then since $\eta = |\ts_i - k_i|, \, q=\frac{\alpha}{2} |\ts_i - k_i| (\ts_i-k_i), \, (i=1,3)$ are convex entropy-entropy flux pair for each $k \in \mathbb{R}$, 
during the same selection of test functions and the limit process as in \cite{Kruzkov_1970},
one can get
\begin{align*}
	& \int_0^1 |\ts_1^*(x,t) - \ts_1^{**}(x,t)| \rmd x \\
	\leq & \int_0^1 |\ts_1^*(x,0) - \ts_1^{**}(x,0)| \rmd x 
	+ \frac{\beta}{4} \int_0^t \int_0^1 \int_{-1}^{1} |\ts_2'(\frac{x+y}{2})| |\ts_3^*(y,\tau) - \ts_3^{**}(y,\tau)| \rmd y \rmd x \rmd \tau 
\end{align*}
and
\begin{align*}
& \int_0^1 |\ts_3^*(x,t) - \ts_3^{**}(x,t)| \rmd x \\
\leq & \int_0^1 |\ts_3^*(x,0) - \ts_3^{**}(x,0)| \rmd x 
+ \frac{\beta}{4} \int_0^t \int_0^1 \int_{-1}^{1} |\ts_2'(\frac{x+y}{2})| |\ts_1^*(y,\tau) - \ts_1^{**}(y,\tau)| \rmd y \rmd x \rmd \tau .
\end{align*}
Adding up these results and using Gronwall's inequality yields the desired uniqueness result
\[
	(\ts_1^*,\ts_3^*)^T = (\ts_1^{**},\ts_3^{**})^T, \quad a.e. \text{       \qedsymbol}
\]




\appendix

\section{Finite Time Blowup of Classical Solutions}\label{app:class}	
For the system \eqref{1.1_sys},
this appendix would provide a proof on the blowup behavior of the classical solutions under the condition that the initial data of the sound waves $\sigma_1$ and $\sigma_3$ are relatively stronger than the steady entropy wave $\sigma_2$. 
Similar to the first section of \cite{Glimm_Lax_1970}, this blowup is essentially caused by the widening effect of the rarefaction waves and can be treated as a continuous version of Proposition \ref{prop:decay}.
\begin{prop}
	For the Cauchy problem \eqref{1.1_sys}, under the assumption \eqref{1.2_1}--\eqref{1.3_sigma2} and $\ts_2 \in C^1$,
	\[ \max_{x \in [0,\frac12)}  |\ts_2'| = \tme   \]
	with $C^1$ initial data $(\ts_{1,0},\ts_{3,0})^T$ satisfying \eqref{1.4_ini} with
	\begin{equation} \label{p1.3_k}
	\frac{\mi}{\tme} > \frac{3}{\ln \frac{13}{12}} \frac{\beta}{\alpha}  
	\end{equation}
	the corresponding classical solution $(\sigma_1,\sigma_3)^T$ would blow up in finite time  $T_b \leq \frac{6}{\alpha \mi}$.
	Furthermore, it must be the geometric blow-up.
\end{prop}


\begin{proof}
Set $K(x)$ as \eqref{2.3_rs2}, then \eqref{2.5_K} holds and
\begin{gather}
\int_{0}^{1} |K(x)| \rmd x \leq \frac{\beta}{4} \tme. \label{p1.3_Kl1}
\end{gather}	
Without loss of generality, one may assume
\[
\TV \sigma_{1,0} \geq \frac{1}{2} \mi, \quad \TV \sigma_{3,0} \leq \frac{1}{2} \mi,
\]
then by the property of periodicity, there exists a point $z^* \in [0,1)$, such that
\begin{equation}
	\frac{\partial \sigma_{1,0}}{\partial x} (z^*) = - \frac{1}{4} \mi. \label{p1.10_ini_max}
\end{equation}

Set
\begin{align}
	w_1 = & \partial_x \sigma_1, \\
	w_3 = & \partial_x \sigma_3. 
\end{align}
Noting that
\[
\frac{\partial}{\partial x} \int_{-1}^{1} K(x+y) \sigma_i(y,t) \rmd y =
- \int_{-1}^{1} K(x+y) w_i(y,t) \rmd y,
\]
one can	deduce from \eqref{1.1_sys} that
\begin{equation} \label{p1.11_sys_w}
	\left\{\begin{aligned}
		& \partial_t w_1 + \alpha \partial_x (\sigma_1 w_1) - \int_{-1}^{1} K(x+y) w_3(y,t) \rmd y = 0,\\
		& \partial_t w_3 - \alpha \partial_x (\sigma_3 w_3) + \int_{-1}^{1} K(x+y) w_1(y,t) \rmd y = 0.
	\end{aligned}\right.
\end{equation}
Multiplying $\mathrm{sgn}(w_1)$ and $\mathrm{sgn}(w_3)$ on both sides of these two equations respectively and integrating over $[0,1)$, noting \eqref{p1.3_Kl1}, one can get
\[
	\frac{\rmd}{\rmd t} \big( \| w_1(\cdot,t) \|_{L^1 [0,1)} + \| w_3(\cdot,t) \|_{L^1 [0,1)} \big) \leq \frac{\beta}{2} \tme \ \big( \| w_1(\cdot,t) \|_{L^1 [0,1)} + \| w_3(\cdot,t) \|_{L^1 [0,1)} \big).
\] 
Thus, 
\begin{equation}\label{p1.12_C0_est}
	\| w_1(\cdot,t) \|_{L^1 [0,1)} + \| w_3(\cdot,t) \|_{L^1 [0,1)} \leq  \mi \exp(\frac{\beta}{2} \tme t),
\end{equation}
which is a continuous version of Proposition \ref{prop:growth}.

Denote the $1$-characteristic passing through $(z,0)$ as $x_1(t;z)$, namely
\begin{equation*}
	\left\{ \begin{aligned}
		& \frac{\rmd x_1(t;z)}{\rmd t} = \alpha \sigma_1(x_1(t;z),t),\\
		& x_1(0;z) = z,
	\end{aligned} \right.
\end{equation*}
then
\begin{equation*}
	\left\{ \begin{aligned}
		& \frac{\rmd}{\rmd t} \frac{\partial x_1(t;z)}{\partial z} = \alpha \frac{\partial \sigma_1(x_1(t;z),t)}{\partial z} \overset{\textrm{def.}}{=} \alpha \sigma_{1,z}(t;z), \\
		& \frac{\partial x_1(0;z)}{\partial z} = 1,
	\end{aligned} \right.
\end{equation*}
and
\begin{equation} \label{p1.14_x1z}
	\frac{\partial x_1(t;z)}{\partial z} = 1 + \alpha \int_0^t \sigma_{1,z}(\tau;z) \rmd \tau.
\end{equation}
By the first equation of the original system \eqref{1.1_sys},
\[
	\frac{\rmd}{\rmd t} \sigma_1(x_1(t;z),t) = - \int_{-1}^1 K( x_1(t;z) + y ) \sigma_3(y,t) \rmd y = - \int_{-1}^1 K(  y ) \sigma_3(y - x_1(t;z),t) \rmd y,
\]
therefore it holds
\begin{equation} \label{p1.11+}
	\frac{\rmd}{\rmd t} \sigma_{1,z}(t;z) = \int_{-1}^{1} K(y) w_3(y - x_1(t;z), t) \frac{\partial x_1(t;z)}{\partial z}  \rmd y.
\end{equation}
Now one may use a bootstrap argument to suppose
\[
	\sigma_{1,z}(t;z^*) \in [-\frac{1}{3} \mi, - \frac{1}{6} \mi], \quad \forall\, t \in [0,T_*)
\]
with
\[
	T_* \leq \min\{\frac{6}{\alpha \mi}, T_b \},
\]
which holds already at $t=0$ by \eqref{p1.10_ini_max}.
Then due to \eqref{p1.14_x1z}
\[
	\left| \frac{\partial x_1(t;z)}{\partial z}\mid_{z=z^*} \right| = \left| 1 + \alpha \int_0^t \sigma_{1,z}(\tau;z^*) \rmd \tau \right| \leq 1, \quad \forall\ t \in [0,T_*).
\]
Thus, integrating \eqref{p1.11+} with respect to $t$ and using \eqref{2.5_K}, \eqref{p1.12_C0_est} and \eqref{p1.3_k}, one may get
\begin{align*}
	\left| \sigma_{1,z}(t;z^*) - \frac{\partial \sigma_1(z,0)}{\partial z} \mid_{z=z^*} \right| \leq 
	& 2 \max_{x\in [0,1)} |K(x)| \ \int_{0}^{t} \| w_3(\cdot,t) \|_{L^1[0,1)} \rmd t \\
	\leq & \frac{\beta}{2} \tme \int_0^t  \mi \exp(\frac{\beta}{2} \tme \tau) \rmd \tau  \\
	= & \mi \big( \exp(\frac{\beta}{2} \tme t) - 1 \big) \\
	< & \frac{1}{12} \mi, \quad \forall\, t \in [0,T_*).
\end{align*}
Therefore, by \eqref{p1.10_ini_max}, one has
\[
	\sigma_{1,z}(t;z^*) \in (- \frac{1}{3} \mi, - \frac{1}{6} \mi), \quad \forall\, t \in [0,T_*),
\]
which completes the bootstrap argument.

Now, suppose by contrary $T_b > \frac{6}{\alpha \mi}$, then by \eqref{p1.14_x1z}, there exists $t^* \leq \frac{6}{\alpha \mi} < T_b $ such that
\begin{equation} \label{A.30}
	\frac{\partial x_1(t;z)}{\partial z} \bigg|_{z=z^*,t=t^*} = 0.
\end{equation} 
Set $x^* = x_1(t^*;z^*)$, then $\frac{\partial \sigma_1}{\partial x}(x^*,t^*)$ is finite since $t^* < T_b$, 
and $\sigma_{1,z}(t^*;z^*)$ is finite and negative by the above bootstrap argument, 
but this contradicts with the chain rule
\[
	\sigma_{1,z}(t;z) = \frac{\partial \sigma_1}{\partial x}(x_1(t;z),t) \frac{\partial x_1(t;z)}{\partial z}. 
\]
In fact, \eqref{A.30} shows that the blow-up is of the geometric type.
\end{proof}

\section{Proof of the lemmas in Section \protect{\ref{ssc:4.2}}} \label{app:B}

In this appendix, the lemmas in Section \ref{ssc:4.2} are proved in details.
The methods used in this part are slight modifications of the ones used in \cite{Glimm_Lax_1970}.
Since the details are quite tedious, and the methods are not new, it can be deleted before publication.

\subsection*{Proof of Lemma \protect{\ref{lem:4.2}}}
By the construction of $\chi^N$, it holds that
\begin{align*}
& | \chi^N (t_1) - \chi^N (t_2) | \leq \alpha \| \ts_1^N \|_{L^\infty} \cdot |t_2 - t_1|, \quad \forall\, t_1,t_2 \in (n \dt, (n+1) \dt ), \\
& | \chi^N( n \dt+) - \chi^N( (n+1) \dt+ ) | = \dx.
\end{align*}
Therefore,
\begin{equation} \label{B.1}
| \chi^N(t_1) - \chi^N(t_2) | \leq \Lambda |t_2 - t_1| + \dx, \quad \forall\, t_1, t_2 \in [T_0, T_0 + T_*],
\end{equation}
where $\Lambda$ is the one for C.F.L. condition.
Now one may just use a diagonal selection method for all the rational time points and take out the uniform convergence subsequence.

Meanwhile, \eqref{B.1} shows that $\Lambda$ is the Lipschitz constant for $\chi$.

\subsection*{Proof of Lemma \protect{\ref{lem:4.3}}}
For each $n_1^N, n_2^N \in \mathbb{Z}^+$ satisfying 
$T_0 \leq n_1^N \dt < n_2^N \dt < T_0 + T_*$,
one can calculate $\chi^N(n_2^N \dt) - \chi^N(n_1^N \dt)$ as follows.
If $\chi^N(t)$ is a line segment for $t \in ((n-1) \dt, n \dt)$
starting from $(m \dx, (n-1) \dt)$, namely,
\[
\chi^N((n-1) \dt) = m \dx,
\]
then its continuation line segment starts from $( (m+1) \dx, n \dt )$, namely,
\[
\chi^N (n \dt+) = (m+1) \dx,
\] 
if and only if the sampling point locates on the left of its end, i.e.,
\begin{equation} \label{B.1+}
(m + \vth_n) \dx < \chi^N(n \dt-).
\end{equation}
Since 
\[
\chi^N (n \dt- ) = \chi^N( (n-1) \dt + ) + \dot\chi^N \dt \geq \chi^N((n-1) \dt) + a \dt = m \dx + a \dt, 
\]
at least for all the cases $\vth_n < a/ \Lambda$ \eqref{B.1+} holds
and the continuation jumps to the right.

Set
\[
S^N(a) = \sharp \{ \vth_n \mid n_1^N \leq n \leq n_2^N, \vth_n < a/\Lambda \}.
\]
There are at least $S^N(a)$ times that $\chi^N$ jumps to the right and at most $n_2^N - n_1^N - S^N(a)$ times to the left.
Thus, 
\[
\chi^N(n_2^N \dt+ ) - \chi^N(n_1^N\dt+) \geq S^N(a) \dx - (n_2 - n_1 - S^N(a)) \dx
\]
and 
\begin{equation} \label{B.2}
\frac{ \chi^N(n_2^N \dt+) - \chi^N(n_1^N \dt+) }{ n_2^N \dt - n_1^N \dt } \geq \frac{\dx}{\dt} (\frac{2 S^N(a)}{n_2^N - n_1^N} -1). 
\end{equation}
Choosing 
\[
n_2^N \dt \to t_2 \; \text{and} \; n_1^N \dt \to t_1, \quad \text{as } N \to \infty,
\]
such that $T_0 \leq t_1 \leq t_2 < T_0+T_*$, 
and using the property that $\{ \vth_n \}$ are independently equi-distributed random variables, one gets that
\[
\lim_{N \to \infty} \frac{2 S^N(a)}{n_2^N - n_1^N} = \frac{a}{\Lambda} +1, \quad a.s..
\]
Thus, taking $N \to \infty$ in \eqref{B.2} and using Lemma \ref{lem:4.2} yields 
\[
\frac{\chi(t_2)-\chi(t_1)}{t_2 - t_1} \geq a, \quad a.s.,
\]
which implies that
\[
\dot{\chi}^N(t) \geq a, \quad a.s.
\]
for all but countable $t$.

Through a similar process one can show the other side of the inequality.
In fact, one has
\begin{lem} \label{lem:B.1}
For two $1$-characteristics $\chi_1$ and $\chi_2$ obtained in the way of Lemma \ref{lem:4.2}, if there exists a constant $c > 0$ and a time span $[t_1,t_2] \subseteq [T_0,T_0+T_*)$, such that
\[
\dot{\chi}^N_1(t) \leq \dot{\chi}^N_2(t) + c, \quad \forall\, t \in [t_1,t_2],
\]
for a subsequence, then
\[
\dot{\chi}_1(t) \leq \dot{\chi}_2(t) + c
\]
for all $t \in [t_1,t_2]$.
\end{lem}
This lemma would be used several times in the proof of Lemma \ref{lem:4.4}.

For the last part of Lemma \ref{lem:4.3}, one can first use the result of Section 4.1 that $\TVchi{}\dot{\chi}^N$ is uniformly bounded.
Noting that $\dot{\chi}^N$ is also uniformly bounded due to Proposition \ref{prop:growth}, 
one can apply Helly's selection principle to get a convergent subsequence such that 
\begin{align*}
& \lim_N \dot{\chi}^N(t) = s(t), \\
& \lim_N \int_{T_0}^{t} |\rmd \dot{\chi}^N| = \bar{s}(t).
\end{align*}

Secondly, one can show that the desired result holds at each continuous point of $\bar{s}(t)$.
In fact, suppose that $\bar{s}(t)$ is continuous at $t_0 \in [T_0,T_0+T_*)$,
then for any given $\varepsilon>0$,
there exist $\theta > 0$ and a subsequence such that
\begin{gather*}
\int_{t_0 - \theta}^{t_0 + \theta} |\rmd \dot{\chi}^N(t)| < \frac{\varepsilon}{3},  \\
| \dot{\chi}^N(t_0) - s(t_0) | < \frac{\varepsilon}{3}, \\
| s(t) - s(t_0) | < \frac{\varepsilon}{3}, \quad \forall\, t \in [t_0 - \theta, t_0 + \theta].
\end{gather*}
Then for $t \in [t_0 - \theta, t_0 + \theta]$, it holds that
\[
s(t) - \varepsilon < \dot{\chi}^N (t) < s(t) + \varepsilon.
\]
It thus follows from the first part of Lemma \ref{lem:4.3} that
\[
s(t) - \varepsilon < \dot{\chi}(t) < s(t) + \varepsilon.
\]
Due to the arbitrariness of $\varepsilon$, the desired result
\[
\dot{\chi}(t_0) = s(t_0) = \lim_N \dot{\chi}^N(t_0)
\]
is proved.

\subsection*{Proof of Lemma \protect{\ref{lem:4.4}}}
To this end, one can define
\[
Q_1^N(\dm^N_{m,n}) = \max\{ 0, -\alpha_{1,m,n} \} \cdot \max\{ 0,- \beta_{1,m,n} \} 
\]
for the $1$-shock collision happened at $\dm^N_{m,n}$ in the approximate solution $(\sigma_1^N,\sigma_3^N)^T$.
Since it counts only the intrafamily wave collision, while each $1$-wave pair ever enters the domain of created in the domain can only collide once,
$\sum_{m,n} Q_1^N(\dm^N_{m,n})$ is uniformly bounded by $(\TV \ts_1^N(\cdot,T_0) + \sum_{m,n} \Delta_1(\dmh) )^2$.
Similar as $C_1^N(\dm^N_{m,n}),\, \Delta_1^N(\dm^N_{m,n})$ and $S^N(\chi^N)$, one may denote
$\rmd C_1^N,\, \rmd \Delta_1^N,\, \rmd S^N$ as well as $\rmd Q_1^N$ as measures assigned their values to the center of the corresponding diamonds.
\footnote{Here the superscript $N$ is added for the quantity corresponding to $(\ts_1^N,\ts_3^N)^T$.}

Then by their uniform bounds, one can get a subsequence of the approximate solutions that
\[
\rmd C_1^N \to \rmd C_1, \; \rmd \Delta^N_1 \to \rmd \Delta_1, \; \rmd S^N(\chi^N) \to \rmd S(\chi), \; \rmd Q_1^N \to \rmd Q_1,
\]
in the $w^*$ topology of measures.

Now for all but countable $t_0 \in  [T_0, T_0+T_*]$,
any of $\rmd C_1, \rmd \Delta_1, \rmd S(\chi) $ and $ \rmd Q_1$ has zero measure at the point $(\chi(t_0),t_0)$.
One can prove that the result of the lemma holds at each of such $t_0$.

Without loss of generality, it is supposed that $\Lambda > 1$, for the $\Lambda$ in the C.F.L. condition.
For a sufficiently small number $\gamma>0$,
one can choose a neighborhood 
\[
V(\gamma) = \{ (x,t) \mid x \in [\chi(t_0) - 20 R, \chi(t_0) + 20R], t\in [t_0 - 20R , t_0 + 20R] \},
\]
where $R = R(\gamma) $ is sufficiently small such that for all  $N$ large enough the following requirements hold
\begin{enumerate}
\item The amount of $1$-waves canceled in $V(\gamma)$ is less than $\gamma^3$, namely, 
\begin{equation} \label{B.3} 
C_1^N(V(\gamma)) < \gamma^3.
\end{equation} 
\item The amount of $1$-shocks entering $\chi^N$ in $V(\gamma)$ is less than $\gamma^3$, namely,
\begin{equation} \label{B.4}
S^N(\chi^N) < \gamma^3.
\end{equation}
\item The amount of $1$-shock collision happened in $V(\gamma)$ is less than $\gamma^3$, namely,
\begin{equation} \label{B.5}
Q^N_1(V(\gamma)) < \gamma^3.
\end{equation}
\item The amount of influence from $3$-waves to $1$-wave in $V(\gamma)$ is less than $\gamma^3$, namely,
\begin{equation} \label{B.6}
\Delta_1^N(V(\gamma)) < \gamma^3.
\end{equation}
\end{enumerate}

First, for any $1$-shock that is strong at some time, one can get an estimate for its strength in the later time.

\begin{fact}\label{fact:1}
Let $\phi^N$ be an approximate characteristic located in $V(\gamma)$ such that $\Str \phi^N(t) > \gamma^{\frac32}$ with $(\phi^N(t),t) \in V(\gamma)$, 
then for any $t' > t$ with $(\phi^N(t'),t) \in V(\gamma)$, it holds that
\begin{gather*}
\gamma^{\frac32} - 2 \gamma^3 \leq \Str \phi^N(t') \leq \Str \phi^N(t) + 3 \gamma^{\frac32}, \\
\mathrm{TV}_{\phi^N \cap V(\gamma)} \dot{\phi}^N \leq 4 \gamma^{\frac32}.
\end{gather*}
\end{fact}
\begin{proof}[Proof of Fact \protect{\ref{fact:1}}]
Due to \eqref{B.3} and \eqref{B.6}, it is direct to get the lower bound.
To get the other side of the estimate, one may estimate the amount of $1$-shocks entering $\phi^N$,
which are $\beta_1, \beta_2, \dots, \beta_k$. 
Due to the lower bound $\Str \phi^N(t') > \frac12 \gamma^\frac32$ obtained already, the $1$-shock collision happened on $\phi^N$ is at least 
\[
\frac12 \gamma^\frac32 \cdot (\sum_j \beta_j).
\]
Thus, due to the upper bound of the total collision \eqref{B.5}, one can get 
\[
\sum_j \beta_j < 2 \gamma^\frac32,
\]
which, combined with \eqref{B.6}, gives the upper bound in the desired estimate and completes the proof for the first part of the Fact. 

Since the change of $\phi^N$ can only be caused by $1$-shock entering, $1$-wave cancellation and $3$-wave influence, one can use \eqref{B.3}, \eqref{B.6} and the result of the first part to get the second estimate in the Fact. 
\end{proof} 

Now one can divide the proof into two parts as follows 
\begin{itemize}
\item[$(A)$] For all $N$ large enough, there is a $1$-shock $\phi^N$ with $\Str \phi^N(t^N) \geq \gamma $ contained in $\sigma^N_1$ and located near $(\chi^N(t_0),t_0)$ 
\begin{gather*}
|\phi^N(t^N) - \chi^N(t_0)| < \frac{\alpha}{4} \gamma R, \\
|t^N - t_0| < \frac{\alpha}{4 \Lambda} \gamma R.
\end{gather*}
\item[$(B)$] $R(\gamma)$ can be further shrunk that for all $N$ large enough, all $1$-shocks of $\sigma_1^N$ located in $V(\gamma)$ have strength less than $\gamma$.\\[2pt]
\end{itemize}

\noindent\emph{Proof of Case $(A)$:} 
In this case, one can prove that $\phi^N(t^N)$ locates actually on $\chi^N$,
and there is a short space interval centered at $\chi^N(t_0)$ such that the total strength of $1$-waves passing through it is small enough,
which implies that the total variation and thus the oscillation of $\ts_1^N$ over this interval is small.
To prove this fact, the idea is that otherwise the $1$-waves would coalesce with $\phi^N$ and cause too strong cancellation at shock collisions.

First, one may show that all the approximate characteristics near $\phi^N$ would roughly point towards it.

\begin{fact} \label{fact:2}
It holds that
\begin{equation} \label{B.5+}
\alpha \sigma_1^N(x,t) + \frac{\alpha}{4} \gamma \leq \dot{\phi}^N(t), \quad \forall\, (x,t) \in (\phi^N(t), \phi^N(t) + \frac{\alpha}{4} \gamma R) \times (t_0 - R/\Lambda, t_0 + R/\Lambda),
\end{equation}
and 
\begin{equation} \label{B.6+}
\alpha \ts_1^N(x,t) - \frac{\alpha}{4} \gamma \geq \dot{\phi}^N(t), \quad \forall\, (x,t) \in (\phi^N(t) - \frac{\alpha}{4} \gamma R, \phi^N(t)) \times (t_0 - R/\Lambda, t_0 + R/\Lambda).
\end{equation}
\end{fact}
\begin{proof}[Proof of Fact \ref{fact:2}] 
Since $\Str \phi^N (t^N) \geq \gamma$, by the previous result, it holds that
\[
\Str \phi^N(t) > \gamma - 2 \gamma^3, \quad \forall\, t > t^N.
\]
Moreover, 
\begin{align}
\dot{\phi}^N(t^N) = & \frac{\alpha}{2} \big( \ts_1^N(\phi^N(t^N)-,t^N) + \ts_1^N(\phi^N(t^N) + , t^N) \big) \notag \\
= & \alpha \ts_1^N(\phi^N(t^N)+,t^N) + \frac{\alpha}{2} \Str \phi^N(t^N) \notag \\
> & \alpha \ts_1^N( \phi^N(t^N)+,t^N ) + \frac{\alpha}{2} \gamma. \label{B.7}
\end{align}
Suppose that on the contrary, for infinitely many $\ts_1^N$ in Case $(A)$, \eqref{B.5+} is violated at $t^* \in (t_0-R,t_0+R)$.
Since $\ts_1^N$ are piece-wise Riemann solutions to the Burgers equation, it is constant along each straight characteristic line.
Without loss of generality, one may suppose $t^*= n^* \Delta t^{N^*} +$ for some $N^*$, then one can choose
\[
x^* = \inf \{ \tilde x \in (\phi^N(t^*), \phi^N(t^*) + \frac{\alpha}{4} \gamma R) \mid \dot{\phi}^N(t^*) < \alpha \ts_1^N(\tilde x,t^*) + \frac{\alpha}{4} \gamma \}.
\]
By this choice, $x^* = m^* \Delta x^{N^*}$ for some $m^*$ and 
\[
\alpha \ts_1^N(x^*-,t^*) + \frac{\alpha}{4} \gamma \leq \dot{\phi}^N(t^*) < \alpha \ts_1^N(x^*+,t^*) + \frac{\alpha}{4} \gamma.
\]
By the construction of $(\ts_1^N, \ts_3^N)^T$, $\ts_1^N$ on 
\[
(x,t) \in (x^* - \dx , x^* + \dx) \times (t^*,t^*+\dt)
\]
would be a centered rarefaction wave.
Thus, there would be an approximate $1$-characteristic $\psi^N$ issuing from $(x^*,t^*)$ such that
\[
\dot{\psi}^N(t^*) = \dot{\phi}^N(t^*) - \frac{\alpha}{4} \gamma.
\]
In what follows, it will be shown that ${\psi}^N(t)$ and $\phi^N(t)$ would coalesce in $V(\gamma)$.
To this end, one may focus on the estimate of $\dot\psi^N$ in the later time $t > t^*$,
during which, the key disturbing factor is the $1$-shocks coming from the domain between $\psi^N$ and $\phi^N$, and entering $\psi^N$.
To deal with this, one may first show that

\begin{claim}
All $1$-shocks crossing $(\phi^N(t^*),\psi^N(t^*)) \times \{t^*\}$ possess a strength less than $\gamma^\frac32$ at least for all the cases that $N$ is large enough.
\end{claim}

\begin{proof}[{Proof of the Claim}] Suppose that on the contrary, $\eta^N(t)$ is an approximate $1$-characteristic such that 
\begin{gather*}
\Str \eta^N(t^*) \geq \gamma^\frac32,\\
\eta^N(t^*) \in (\phi^N(t^*),\psi^N(t^*)).
\end{gather*}
Now it is hoped that $\eta^N$ would coalesce with $\phi^N$ in $V(\gamma)$.
In fact, by the previous result
\begin{gather*}
\Str \eta^N(t) \geq \gamma^\frac32 - 2 \gamma^3, \quad \forall\, t > t^*,\\
\mathrm{TV}_{\eta^N} \dot{\eta}^N \leq \gamma^\frac32, 
\end{gather*}
while, by the definition of $x^*$ and $\psi^N$,
\begin{gather*}
\dot{\eta}^N(t^*) + \frac{\alpha}{4} \leq \dot{\phi}^N(t), \quad 
\eta^N(t^*) \in (\phi^N(t^*), \phi^N(t^*)+ \frac{\alpha}{4} \gamma R).
\end{gather*}
Therefore, 
\[
\dot{\eta}^N(t) - \gamma^\frac32 + \frac{\alpha}{4} \gamma \leq \dot{\phi}^N(t).
\]
As in Lemma \ref{lem:B.1}, almost surely for $\vth$, $\eta^N$ would coalesce with $\dot{\phi}^N$ in $V(\gamma)$ at least for large enough $N$, which would cause a shock collision of size at least $\frac{1}{2} \gamma^\frac52$.
This contradicts with one selection requirement of $V(\gamma)$, \eqref{B.5}, and completes the proof of the Claim. 
\end{proof} 

Furthermore, one has

\begin{claim}
The strength for any $1$-shock crossing $(\phi^N(t),\psi^N(t)) \times \{t\}$ for $t \geq t^*$ is less than $4 \gamma^\frac32$, at least for large enough $N$.
\end{claim}

\begin{proof}[Proof of the Claim]
Denote $\xi_n$ as the strength of the $1$-shock between $\phi^N(n\dt+ )$ and $\psi^N(n\dt+)$ at $n \dt +$, then
\begin{equation} \label{B.8}
\xi_n \leq \xi_{n-1} + \Delta_1^N(\dm^N_{m,n})
\end{equation}
for some $m$.
Suppose that by contrary, $n'\dt+ > t^*$ is the first time such that
\[
\xi_{n'} > 4 \gamma^\frac32.
\]
Then due to the result of the last Claim and \eqref{B.1}, \eqref{B.6} as well as \eqref{B.8}, it holds that
\[
\xi_{n'} < 8 \gamma^\frac32 + 2 \gamma^3,
\]
and there exists $n'',n'''$ with $n^* < n''' < n'' < n'$ such that
\begin{align}
3 \gamma^\frac32 < \xi_n \leq 4 \gamma^\frac32, \quad & \forall\, n'' \leq n < n',\\
2 \gamma^\frac32 < \xi_n \leq 3 \gamma^\frac32, \quad & \forall\, n''' \leq n < n''.
\end{align}
Then for each $n''' \leq n \leq n'$, there exists a diamond $\dm^N_{m,n}$ that produces the wave of strength $\xi_n$, namely,
\[
-\xi_n = \alpha_{m,n} + \beta_{m,n} + \delta \Delta^N_1(\dm^N_{m,n})
\]
with
\[
Q_1^N(\dm^N_{m,n}) = \max \{0,-\alpha_{m,n}\} \cdot \max\{0,-\beta_{m,n} \}
\]
\begin{enumerate}
\item For $\alpha_{m,n} < 0, \beta_{m,n} < 0$, without loss of generality,
one may assume that $|\alpha_{m,n}| > |\beta_{m,n}|$, then
\begin{align*}
\xi_{n-1} \geq & |\alpha_{m,n}| \geq \frac12 (|\alpha_{m,n}| + |\beta_{m,n}|) = \frac12 (\xi_n+\delta \Delta_1^N(\dm^N_{m,n})) \\
\geq & \frac12 (\xi_n - \Delta_1^N(\dm^N_{m,n})).
\end{align*}
Thus,
\[
Q_1^N(\dm^N_{m,n}) = |\alpha_{m,n}| |\beta_{m,n}| \geq |\alpha_{m,n}| \cdot ( \xi_n - |\alpha_{m,n}| - \Delta_1^N(\dm^N_{m,n}) ).
\]
By the property of the parabola and the above bounds of $\alpha_{m,n}$, it holds that
\[
Q_1^N(\dm^N_{m,n}) \geq \xi_{n-1} (\xi_n - \xi_{n-1} - \Delta_1^N(\dm^N_{m,n})).
\]
Then \eqref{B.8} leads to
\[
3 \xi_{n-1} \geq \xi_n + \xi_{n-1} - \Delta_1^N(\dm^N_{m,n}).
\]
Thus,
\begin{equation} \label{B.9.1}
3 Q_1^N(\dm^N_{m,n}) \geq (\xi_n + \xi_{n-1} - \Delta_1^N(\dm^N_{m,n})) (\xi_n - \xi_{n-1} - \Delta_1^N(\dm^N_{m,n}))
\end{equation}
\item For $\beta_{m,n} \geq 0$, due to the choice of $n'''$, it holds that $\xi_n > 15 \gamma^\frac32$.
Then \eqref{B.6} implies that
\[
- \xi_{n-1} < \alpha_{m,n} < 0.
\]
Therefore,
\[
\xi_n - \Delta_1^N(\dm^N_{m,n}) \leq - \alpha_{m,n} - \beta_{m,n} \leq - \alpha_{m,n} < \xi_{n-1},
\]
and 
\begin{equation} \label{B.9.2}
0 \geq (\xi_n + \xi_{n-1} - \Delta_1^N(\dm^N_{m,n})) (\xi_n - \xi_{n-1} - \Delta_1^N(\dm^N_{m,n}))
\end{equation}
\end{enumerate}
Adding up \eqref{B.9.1}--\eqref{B.9.2} for all $n''' < n \leq n'$ and noting that $\xi_{n'}$ is the largest one over $\xi_n$ and \eqref{B.6}, one can get
\begin{align*}
3 Q_1^N(V(\gamma)) \geq & \xi_{n'}^2 - \xi_{n'''}^2 - \xi_{n'} \gamma^3 \\
\geq & 7 \gamma^3 - 8 \gamma^\frac92 > 4 \gamma^3,
\end{align*}
which contradicts with \eqref{B.5} and completes the proof of the Claim. 
\end{proof}

Now one can calculate $\psi^N(t)$ for $t > t^*$.
Denote $\alpha_1, \dots, \alpha_k$ as the sequence of $1$-shocks entering $\psi^N$ from left before the possible coalesce of $\psi^N$ and $\phi^N$. Then at the time $\alpha_j$ entering $\psi^N$, the shock collision would be at lest 
\[
|\alpha_j| (\sum_{p<j} |\alpha_p| - c_j ),
\]
where $c_j$ denotes all the cancellation and $3$-wave influence before the entering.
Summing up this estimate at each entering time gives
\[
\gamma^3 > Q_1^N(V(\gamma)) \geq \frac12 (\sum_j |\alpha_j|)^2 - \frac12 (\sum_j |\alpha_j|^2 ) - 2 \gamma^3 \sum_j |\alpha_j|
\]
By the previous result, $\max |\alpha_j| \leq 4 \gamma^\frac32$, therefore
\[
(\sum_j |\alpha_j|)^2 - (2 \gamma^\frac32 + 2 \gamma^3)(\sum_j |\alpha_j|) < \gamma^3,
\]
which yields that
\[
\sum_j |\alpha_j| < 2 \gamma^\frac32.
\]
Meanwhile, the only factors that can increase $\dot{\psi}^N$ are the cancellation from right and $3$-wave influence on it and the $1$-shocks entering from left, 
all of which are bounded by $2 \gamma^3 + 2 \gamma^\frac32$.
Thus,
\begin{align*}
\dot{\psi}^N(t) > & \dot{\phi}^N(t^*) - \frac{\alpha}{4} \gamma + 3 \gamma^\frac32 \\
\geq & \dot{\phi}^N(t) - \frac{\alpha}{4} \gamma + 4 \gamma^\frac32, \quad \forall\, t > t^*.
\end{align*}
Since
\[
\psi^N(t^*) = x^* \in (\phi^N(t^*), \phi^N(t^*) + \frac{\alpha}{4} \gamma R),
\]
as in Lemma \ref{lem:B.1}, almost surely for $\vth$, $\phi^N$ and $\psi^N$ would coalesce in $V(\gamma)$ at a time $\hat{t}^N < 2 R \Lambda$ at least for large enough $N$.

For the $1$-rarefaction waves passing through $(\phi^N(t^*), \psi^N(t^*)) \times \{t^*\}$, since no $1$-rarefaction wave can cross an approximate $1$-characteristic, 
they would be demolished before $\hat{t}^N$ by either the cancellation or the $3$-wave influence, which implies that their total strength is at most $2 \gamma^3$.
Therefore,
\[
\dot{\psi}^N(t^*) \leq \alpha \ts^N_1(\phi^N(t^*)+, t^*+) + 2 \gamma^3 < \dot{\phi}^N(t^*) - \frac{\alpha}{2} \gamma + 2 \gamma^3,
\]
which contradicts with the definition of $\psi^N$ and completes the proof of the Fact.
\end{proof} 

Due to this fact and our assumption in Case $(A)$, $\phi^N$ and $\chi^N$ would coalesce in $V(\gamma)$ almost surely at least for large enough $N$
which would cause a $1$-shock entering of $\chi^N$ with strength at least $\gamma - 2 \gamma^3$, which is forbidden by \eqref{B.4}.
The only possibility is that $\phi^N$ locates exactly on $\chi^N$ at $t^*$.

Since the $1$-shock entering $\chi^N$, the $1$-wave cancellation happening on $\chi^N$ and the $3$-wave influence acting on $\chi^N$ can be bounded together by $3 \gamma^3$, it holds that
\[
\TVchi{} {\dot{\chi}^N} < 3 \gamma^3
\]
in $V(\gamma)$.

Now one can apply \eqref{B.5+}--\eqref{B.6+} at $t = t_0$ to get that all the $1$-shocks passing through
$( \chi^N(t_0) - \frac{\alpha}{4} \gamma R, \chi^N(t_0) + \frac{\alpha}{4} \gamma R  ) \times \{t_0\}$ 
would either be demolished by cancellation or $3$-wave influence, or enter $\chi^N$ in $V(\gamma)$ almost surely at least for large enough $N$, 
and all $1$-rarefaction waves would be demolished by cancellation or $3$-wave influence.
Thus, the total variation of $\ts_1^N$ along this interval can be bounded by $O(\gamma^3)$ which gives the desired equicontinuity result.\\[2pt]

\noindent\emph{Proof of Case $(B)$:}
The main idea in this part of the proof is that as it has been assumed that each $1$-shock in $V(\gamma)$ possesses a strength less than $\gamma$, 
while all the cancellation and $3$-wave influence are at most $2 \gamma^3$, 
using the $\gamma^3$ bound of the collision, one can show that 
the total variation of the speed of the approximate $1$-characteristics can only be $4 \gamma$, thus for each approximate $1$-characteristic near $\dot{\chi}^N$, if its speed is different from $\chi^N$ with a large order,
it would enter $\chi^N$ in the future,
which gives a small bound for rarefaction waves to cause the assumed speed difference,
or in the past, which is forbidden by  the construction process of the approximate characteristics.
Then one can estimate the oscillation of the characteristic speed and thus the oscillation of the solution.

As the first ingredient of the proof, one may show that

\begin{fact} \label{fact:3}
For each approximate $1$-characteristic $\phi^N$, the total strength of $1$-shocks entering it in $V(\gamma)$ is no more than $4 \gamma$.
\end{fact}
\begin{proof}[Proof of Fact \ref{fact:3}] 
The proof is quite similar to the one for Fact \ref{fact:1} given above.

Suppose, on the contrary, that the total strength is stronger than $4\gamma$.
Since the cancellation and $3$-wave influence can be bounded by $2\gamma^3$, while each $1$-shock is weaker than $\gamma$, 
one can get a time $t^*$ that 
\[
\frac32 \gamma < \Str \phi^N(t^*) < \frac52 \gamma + \gamma^3, 
\]
and there would be a sequence of $1$-shocks $\alpha_1, \dots, \alpha_k$ with total strength at least $\frac32 \gamma - 3 \gamma^3$ that enters $\phi^N$ in $V(\gamma)$ later than $t^*$.
Due to the previous result,
\[
\Str \phi^N(t) > \frac32 \gamma - 2 \gamma^3, \quad \forall\, t > t^*.
\]
Then the collision can be estimated from below as
\[
Q_1^N(V(\gamma)) \geq (\frac32 \gamma - 2 \gamma^3) \cdot (\sum_j |\alpha_j|) > 3 \gamma^2,
\]
which contradicts with \eqref{B.6} and completes the proof of Fact \ref{fact:3}. 
\end{proof} 

Using Fact \ref{fact:3} and the bounds of cancellation and $3$-wave influence, one can conclude that
\begin{equation} \label{B.10}
\mathrm{TV}_{\phi^N \cap V(\gamma)} \dot{\phi}^N < 4 \gamma + 2 \gamma^3,
\end{equation}
for any approximate $1$-characteristic $\phi^N$.

As the second ingredient of the proof, one can prove that all approximate $1$-characteristics would roughly point away from each other.

\begin{fact} \label{fact:4}
For any given $t' \leq t_0$ with $t' = n_1 \dt +$ for some $n_1$, and $x_1,x_2$ with
\[
\chi^N(t') \leq x_1 < x_2 < \chi^N(t') + 24 \gamma R,
\]
it holds that
\begin{equation} \label{B.11}
\dot{\phi}^N_{x_1,t'}(t') \leq \dot{\phi}^N_{x_2,t'}(t') + 13 \gamma,
\end{equation}
where $\phi^N_{x,t}$ is an arbitrary approximate $1$-characteristic issuing from $(x,t)$.
\end{fact}
\begin{proof}[Proof of Fact \ref{fact:4}] 
Suppose by contrary that \eqref{B.11} is violated, then by \eqref{B.10} 
\[
\dot{\phi}^N_{x_1,t'}(t) < \dot{\phi}^N_{x_2,t'}(t) + 4 \gamma, \quad \forall\, t> t'.
\]
Since the initial distance is less than $24 \gamma R$, almost surely for large enough $N$, 
$\phi^N_{x_1,t'}$ and $\phi^N_{x_2,t'}$ would coalesce in $V(\gamma)$.

All $1$-shocks passing through $[x_1,x_2] \times \{t'\}$ would either be demolished by the cancellation or $3$-wave influence, or entering $\phi^N_{x_1,t'}$ or $\phi^N_{x_2,t'}$.
Therefore due to Fact \ref{fact:3}, their total strength can be bounded by $8 \gamma + 2 \gamma^3$.

Meanwhile, $\dot{\phi}^N_{y,t'}$ increases with $y$ going from $x_1$ to $x_2$ only at the time when $y$ passing a $1$-shock.
Thus,
\[
\dot{\phi}^N_{x_1,t'}(t') < \dot{\phi}^N_{x_2,t'}(t') + 9 \gamma,
\]
which contradicts with the assumption at the beginning of the proof. 
\end{proof} 

Using Fact \ref{fact:4}, taking $t' = t_0$ and $x_1 = \chi^N(t_0)$, one can get
\[
\alpha \ts_1^N(y,t_0) = \dot{\phi}^N_{y,t_0}(t_0) \geq \dot{\chi}^N(t_0) - 13\gamma, \quad \forall\, \chi^N(t_0) < y < \chi^N(t_0) + 24\gamma R,
\]
which provides a lower bound of $\ts_1^N$ on the right of $\chi^N$.

To show the other side of the result, a rough idea is that if an approximate $1$-characteristic moves much faster than $\chi^N$ at $t_0$, then due to \eqref{B.10}, it always moves much faster and would cross $\chi^N$ in the past which is forbidden by the construction of the approximate characteristics.
But in our method, it is not always plausible to find the previous part for one given approximate characteristic.
Thus one should change the strategy to an equivalent one that one may search for one approximate characteristic issuing from a point earlier than $t_0$ that reaches the interval $[\frac12 \gamma R. 10 \gamma R]$ at $t_0$ with a desired speed, 
which in fact is to find the previous part for one of the approximate characteristics at  $[\frac12 \gamma R. 10 \gamma R] \times \{t_0 \} $.
Then the speed of this characteristic can provide an upper bound for all others on its left due to \eqref{B.11}.

To be more specific, one may look at one of the approximate $1$-characteristic $\phi^N_{y,R/\Lambda} $ passing through $(y,t_0- R/\Lambda)$ for 
\[
y \in [\chi^N(t_0- R/\Lambda), \chi^N(t_0-R/\Lambda) + 24 \gamma R]
\]
and denote a straight line 
\[
\Phi^N_y(t) = \dot{\phi}^N_{y,t_0-R/\Lambda} ( t - t_0 + R/\Lambda ).
\]
Taking $x_1 = \chi^N(t_0-R/\Lambda), x_2 = y$ and $t' = t_0 - R / \Lambda$ in \eqref{B.11}, one can get
\[
\dot{\chi}^N(t_0 - R/ \Lambda) \leq \dot{\phi}^N_{y,t_0-R/\Lambda} (t_0 - R/ \Lambda) + 13 \gamma.
\]
Due to \eqref{B.10} for $\phi^N = \chi^N$, it holds
\[
\chi^N(t_0) - \chi^N(t_0 - R/ \Lambda) - 5 \gamma R \leq \Phi^N_y(t_0) - y + 13 \gamma R.
\]
Thus, 
\[
\Phi^N_y(t_0) \geq \chi^N(t_0) - 18 \gamma R + (y - \chi^N(t_0 - R/ \Lambda) ).
\]
Noting that
\[
\Phi^N_{\chi^N(t_0-R/\Lambda)} \leq \chi^N(t_0) + 4 \gamma R + 2 \gamma^3 R 
\]
and that as $y$ increases continuously at the time crossing a $1$-rarefaction wave, while decreases sharply at the time crossing a $1$-shock, there exists $y^N \in [\chi^N(t_0-R/\Lambda), \chi^N(t_0 - R/\Lambda + 24 \gamma R)]$,
such that
\[
\Phi^N_{y^N}(t_0) = \chi^N(t_0) + 5 \gamma R.
\]
By \eqref{B.10} for $\phi^N_{y^N}$ and the definition of $\Phi^N_{y^N}$, one can get
\[
|\phi^N_{y^N}(t_0) - \Phi^N_{y^N}(t_0)| < 4 \gamma R + 2 \gamma^3 R,
\]
and 
\begin{equation} \label{B.12}
\chi^N(t_0) + \frac12 \gamma R < \phi^N_{y^N}(t_0) < \chi^N(t_0) + 10 \gamma R.
\end{equation}
Then one may get an upper bound for $\dot{\phi}^N_{y^N}$, 
for which one needs only to get an average speed for $t \in [t_0 - R/ \Lambda,t_0] $ and apply \eqref{B.10}.
In fact, by the second inequality of \eqref{B.12},
and that $\phi^N$ lies on the right of $\chi^N$, it holds that
\[
\frac{\phi^N_{y^N}(t_0) - \phi^N_{y^N}(t_0 - R/ \Lambda)}{R / \Lambda} < \frac{\chi^N(t_0) - \chi^N(t_0-R/ \Lambda)}{R / \Lambda} + 10 \gamma.
\]
Thus, 
\[
\dot{\phi}^N_{y^N} (t_0) < \dot{\chi}^N(t_0) + 19 \gamma.
\]
Now for each approximate $1$-characteristic $\psi^N$ passing through $[\chi^N(t_0), \chi^N(t_0) + \frac12 \gamma R]$, 
since it lies on the left of $\phi^N_{y^N}$, using \eqref{B.11} for $x_1 = \psi^N(t_0), x_2 = \phi^N(t_0)$ and $t=t_0$ leads to
\[
\dot{\psi}^N(t_0) \leq \dot{\phi}^N(t_0) + 13 \gamma < \dot{\chi}^N(t_0) + 32 \gamma,
\]
which yields an upper bound for the speed of all characteristics, and thus for $\ts_1^N$, near $\chi^N$.
Similar results can be proved on the left of $\chi^N$.
This completes the proof of Case $(B)$ and the lemma. \qedsymbol

\subsection*{Proof of Lemma \protect{\ref{lem:4.5}}: }
As a direct application of Lemma \ref{lem:4.4}, one can get

\begin{cor}
There exists a subsequence of approximate solutions such that for all but countable $t$,
\[
\lim_N \ts_1^N( \chi^N(t) \pm 0, t) = \ts_1 (\chi(t) \pm 0, t).
\] 
\end{cor}
\emph{Proof of the Corollary:}
In Section 4.1, it is proved that $\TVchi{} \ts_1^N(\chi^N(\cdot) \pm 0, \cdot)$ is uniformly bounded.
Thus, by Helly's selection principle, there exists a subsequence such that
\[
\lim_N \ts_1^N(\chi^N(t) \pm 0, t) = U^\pm (t),
\]
for some BV function $U^\pm$.

Now, one can complete the proof just by noting that
\begin{align*}
& | U^\pm(t) - \ts_1(\chi(t) \pm y,t ) | \\
\leq & |U^\pm(t) - \ts_1^N(\chi^N(t) \pm 0, t)| \\
& + | \ts_1^N(\chi^N(t) \pm 0,t) - \ts_1^N( \chi(t) \pm y, t) | \\
& + | \ts_1^N(\chi(t) \pm y, t) - \ts_1(\chi(t) \pm y,t) |,
\end{align*}
where on the right hand side, for each $\varepsilon>0$, the first term is bounded by $\varepsilon/3$ for all large enough $N$ in the subsequence,
the second term is bounded by $\varepsilon/3$ for any small enough $y$ due to Lemma \ref{lem:4.4},
while the third one is bounded by $\varepsilon/3$ for almost all $y$, since $\| \ts_1^N - \ts_1 \|_{L^1} \to 0$ as $N \to \infty$. \qedsymbol

With this Corollary in hand, one can take $N \to \infty$ for the Rankine-Hugoniot condition in $\ts_1^N$ and completes the proof.

\section{Approximate Characteristics}
In this appendix, some details in the construction of the approximate characteristics are explained.
And based on these construction, some estimates, especially \eqref{4.*1}--\eqref{4.**3} for half diamonds, are checked.

First, as in Page 30 of \cite{Glimm_Lax_1970}, there are roughly $16$ cases,  in each of which one should assign the continuation of a line segment in an approximate $1$-characteristic. 
See Figure \ref{fig:C.1} for eight of them that the approximate $1$-characteristic, which is marked by dashed lines, enters from the southeast edges. 
The other eight cases can be analyzed similarly.
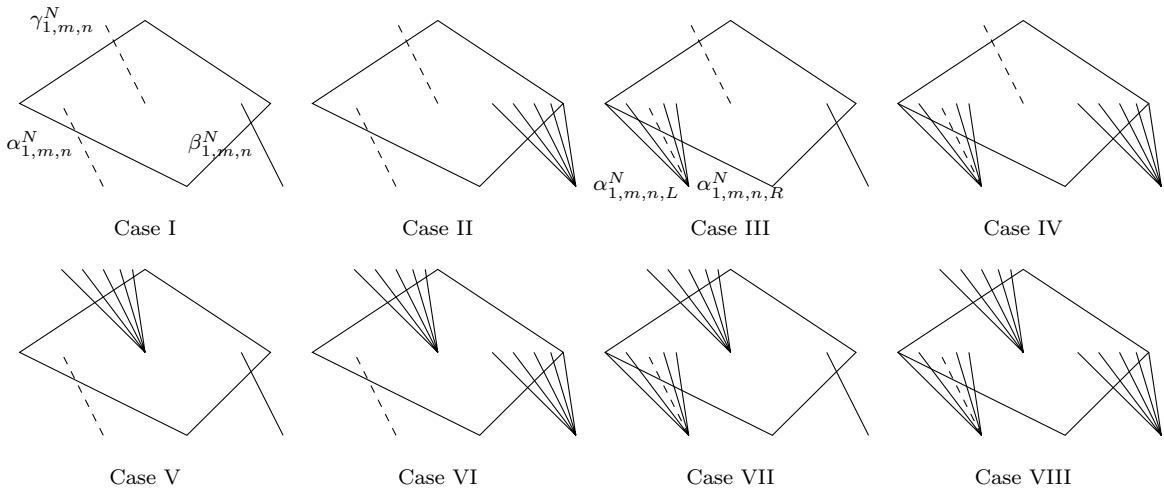
\begin{figure}[htbp]
\scriptsize
\centering
\begin{tikzpicture}[scale = .55]

    \draw[dashed] (0,0) -- ++(-1,2) node[pos=1,left] {$\gamma^N_{1,m,n}$};
    \draw[dashed] (-1,-2) -- ++(-1,2) node[pos=0.5,left] {$\alpha^N_{1,m,n}$}; 
    \draw (3.3,-2) --  ++(-1,2) node[pos=0.5,left] {$\beta^N_{1,m,n}$};
    \draw (-3,0) -- (1,-2) -- (3,0) -- (0,2) -- cycle;
    \node (I) at (0,-3) {Case I};

    \draw[dashed] (7,0) -- ++(-1,2);
    \draw[dashed] (6,-2) -- ++(-1,2); 
    \draw (10.3,-2) -- +(-1,2) (10.3,-2)-- +(-2,2) (10.3,-2)-- +(-1.5,2)  (10.3,-2)-- +(-0.6,2) (10.3,-2)-- +(-0.3,2);
    \draw (4,0) -- (8,-2) -- (10,0) -- (7,2) -- cycle;
    \node (II) at (7,-3) {Case II};

    \draw[dashed] (14,0) -- ++(-1,2);
    \draw[dashed] (13,-2) -- ++(-1,2); 
    \draw (13,-2) -- +(-2,2) node[pos=0,left] {$\alpha^N_{1,m,n,L}$} (13,-2) -- +(-1.5,2) (13,-2) -- +(-0.6,2) (13,-2) -- +(-0.3,2) node[pos=0,right] {$\alpha^N_{1,m,n,R}$};   
    \draw (17.3,-2) --  ++(-1,2);
    \draw (11,0) -- (15,-2) -- (17,0) -- (14,2) -- cycle;
    \node (III) at (14,-3) {Case III};

    \draw[dashed] (21,0) -- ++(-1,2);
    \draw[dashed] (20,-2) -- ++(-1,2); 
    \draw (20,-2) -- +(-2,2) (20,-2) -- +(-1.5,2) (20,-2) -- +(-0.6,2) (20,-2) -- +(-0.3,2);   
    \draw (24.3,-2) -- +(-1,2) (24.3,-2)-- +(-2,2) (24.3,-2)-- +(-1.5,2)  (24.3,-2)-- +(-0.6,2) (24.3,-2)-- +(-0.3,2);
    \draw (18,0) -- (22,-2) -- (24,0) -- (21,2) -- cycle;
    \node (IV) at (21,-3) {Case IV};

    \draw (0,-6) -- +(-1,2) (0,-6)-- +(-2,2) (0,-6)-- +(-1.5,2)  (0,-6)-- +(-0.6,2) (0,-6)-- +(-0.3,2);
    \draw[dashed] (-1,-8) -- ++(-1,2); 
    \draw (3.3,-8) --  ++(-1,2);
    \draw (-3,-6) -- (1,-8) -- (3,-6) -- (0,-4) -- cycle;
    \node (V) at (0,-9) {Case V};

    \draw (7,-6) -- +(-1,2) (7,-6)-- +(-2,2) (7,-6)-- +(-1.5,2)  (7,-6)-- +(-0.6,2) (7,-6)-- +(-0.3,2);
    \draw[dashed] (6,-8) -- ++(-1,2); 
    \draw (10.3,-8) -- +(-1,2) (10.3,-8)-- +(-2,2) (10.3,-8)-- +(-1.5,2)  (10.3,-8)-- +(-0.6,2) (10.3,-8)-- +(-0.3,2);
    \draw (4,-6) -- (8,-8) -- (10,-6) -- (7,-4) -- cycle;
    \node (VI) at (7,-9) {Case VI};

    \draw (14,-6) -- +(-1,2) (14,-6)-- +(-2,2) (14,-6)-- +(-1.5,2)  (14,-6)-- +(-0.6,2) (14,-6)-- +(-0.3,2);
    \draw[dashed] (13,-8) -- ++(-1,2); 
    \draw (13,-8) -- +(-2,2) (13,-8) -- +(-1.5,2) (13,-8) -- +(-0.6,2) (13,-8) -- +(-0.3,2);   
    \draw (17.3,-8) --  ++(-1,2);
    \draw (11,-6) -- (15,-8) -- (17,-6) -- (14,-4) -- cycle;
    \node (VII) at (14,-9) {Case VII};

    \draw (21,-6) -- +(-1,2) (21,-6)-- +(-2,2) (21,-6)-- +(-1.5,2)  (21,-6)-- +(-0.6,2) (21,-6)-- +(-0.3,2);
    \draw[dashed] (20,-8) -- ++(-1,2); 
    \draw (20,-8) -- +(-2,2) (20,-8) -- +(-1.5,2) (20,-8) -- +(-0.6,2) (20,-8) -- +(-0.3,2);   
    \draw (24.3,-8) -- +(-1,2) (24.3,-8)-- +(-2,2) (24.3,-8)-- +(-1.5,2)  (24.3,-8)-- +(-0.6,2) (24.3,-8)-- +(-0.3,2);
    \draw (18,-6) -- (22,-8) -- (24,-6) -- (21,-4) -- cycle;
    \node (VIII) at (21,-9) {Case VIII};

\end{tikzpicture}
\caption{Eight cases of waves in one mesh diamond}
\label{fig:C.1}
\end{figure}

For Cases I--IV, where $\gammah < 0$, one has no other choice but to choose the leaving $1$-shock as the continuation of the approximate $1$-characteristic.
Meanwhile, \eqref{4.*1}--\eqref{4.**3} can be shown as follows.

\subsection*{Case I} ($\alpha^N_{1,m,n} < 0, \beta^N_{1,m,n} < 0, \gamma^N_{1,m,n} < 0$): 
In this case, 
\begin{gather*}
E_1^+(\dm^N_{m,n}) = 0, \quad  E_1^-(\dm^N_{m,n}) = \alpha^N_{1,m,n} + \beta^N_{1,m,n}, \\
L_1^+(\dm^N_{m,n}) = 0, \quad  L_1^-(\dm^N_{m,n}) = \gamma^N_{m,n}, \\
C_1(\dm^N_{m,n}) = 0, \quad  \Delta_1(\dm^N_{m,n}) = |\gamma^N_{1,m,n} - \alpha^N_{1,m,n} - \beta^N_{1,m,n}|,
\end{gather*}
and 
\begin{gather*}
E_1^\pm(\dml) = L_1^+(\dml) = S_1(\dml) = C_1^\pm(\dml) = \Delta_1(\dml) = 0, \\
E_1^+(\dmr) = L_1^+(\dmr) = C_1^+(\dmr) = 0, \\
E_1^-(\dmr) = \beta^N_{1,m,n}, \quad S_1(\dmr) = \gamma^N_{1,m,n} - \alpha^N_{1,m,n}, \\
C_1^-(\dmr) = 0, \quad \Delta_1(\dmr) = \Delta_1(\dmh).
\end{gather*}
Then it is easy to check \eqref{4.*1}--\eqref{4.**3}.

\subsection*{Case II} 
($\alphah < 0, \betah > 0, \gammah<0$):
In this case, 
\begin{gather*}
E_1^+(\dmh) = \betah, \quad E_1^-(\dmh) = \alphah, \\
L_1^+(\dmh) = 0, \quad L_1^-(\dmh) = \gammah,\\
C_1(\dmh) = \min\{ -\alphah, \betah \}, \quad \Delta_1(\dmh) = \big| |\gammah | - |\alphah| + |\betah| \big|,
\end{gather*}
and 
\begin{gather*}
E_1^\pm(\dml) = L_1^+(\dml) = S_1(\dml) = C_1^\pm(\dml) = 0, \\
E_1^-(\dmr) = C_1^-(\dmr) = 0, \\
E_1^+(\dmr) = \betah, \quad L_1^+(\dmr) = 0, \\
C_1^+(\dmr) = C_1(\dmh), \quad \Delta_1(\dmr) = \Delta_1(\dmh), \\
S_1(\dmr) = \min\{ 0, \gammah - \alphah \}
\end{gather*}
Now \eqref{4.*1}--\eqref{4.**3} follow easily.

\subsection*{Case III}
($\alphah = \alphal + \alphar > 0, \betah < 0, \gammah < 0$):
In this case,
\begin{gather*}
E_1^+(\dmh) = \alphah, \quad E_1^-(\dmh) = \betah, \\
L_1^+(\dmh) = 0, \quad L_1^-(\dmh) = \gammah, \\
C_1(\dmh) = \min \{ \alphah, - \betah \}, \quad \Delta_1(\dmh) = \big| \gammah + \alphah - \betah \big|,
\end{gather*}
and
\begin{gather*}
E_1^+(\dml) = \alphal, \quad E_1^-(\dml) = 0, \\
L_1^+(\dml) =0, \quad S_1(\dml) = 0, \quad C_1^-(\dml) = 0, \\
E-1^+(\dmr) = \alphar, \quad E_1^-(\dmr) = \betah, \\
L_1^+(\dmr) = 0, \quad S_1(\dmr) = \gammah.
\end{gather*}
\subsubsection*{Subcase III.I}($-\betah > \alphah$):
\begin{gather*}
C_1(\dmh) = \alphah \\
C_1^+(\dml) = \alphal, \quad \Delta_1(\dml) = 0, \\
C_1^+(\dmr) = \alphar, \quad C_1^-(\dmr) = \alphah, \quad \Delta_1(\dmr) = \Delta_1(\dmh).
\end{gather*}
Now \eqref{4.*1}--\eqref{4.**3} can be easily checked.
\subsubsection*{Subcase III.II}($\alphah \geq - \betah > \alphar$):
\begin{gather*}
C_1(\dmh) = - \betah, \\
C_1^+(\dml) = - \betah - \alphar, \quad \Delta_1(\dml) = \alphah + \betah, \\
C_1^+(\dmr) = \alphar, \quad C_1^-(\dmr) = - \betah, \quad \Delta_1(\dmr) =  - \gammah,
\end{gather*}
and \eqref{4.*1}--\eqref{4.**3} are easy to be checked.
\subsubsection*{Subcase III.III}($ \alphar \geq - \betah $):
\begin{gather*}
C_1(\dmh) = - \betah, \\
C_1^+(\dml) = 0, \quad \Delta_1(\dml) = \alphal,\\
C-1^\pm(\dmr) = - \betah, \quad \Delta_1(\dml) = - \gammah + \alphar - \betah
\end{gather*}
and \eqref{4.*1}--\eqref{4.**3} can be easily checked.

\subsection*{Case IV}($\alphah = \alphal + \alphar > 0, \betah > 0, \gammah < 0$):
In this case,
\begin{gather*}
E_1^+(\dmh) = \alphah + \betah, \quad E_1^-(\dmh) = 0, \\
L_1^+(\dmh) = 0, \quad L_1^-(\dmh) = \gammah, \\
C_1(\dmh) = 0, \quad \Delta_1(\dmh) = - \gammah + \alphah + \betah,
\end{gather*}
and 
\begin{gather*}
E_1^+(\dml) = \alphal, \quad L_1^+(\dml) = 0, \\
E_1^-(\dml) = 0, \quad S_1(\dml) = 0, \\
C_1^\pm(\dml) = 0, \quad \Delta_1(\dml) = \alphal, \\
E_1^+(\dmr) = \alphar+\betah, \quad L_1^+(\dmr) = 0, \\
E_1^-(\dmr) = 0, \quad S_1(\dmr) = \gammah, \\
C_1^\pm(\dmr) = 0, \quad \Delta_1(\dmr) = - \gammah + \alphar + \betah.
\end{gather*}
Then it is easy to check \eqref{4.*1}--\eqref{4.**3}.\\[2pt]

For Cases V--VIII, one needs to be careful to assign the continuation line segment to ensure \eqref{4.*1}--\eqref{4.**3} as well as the requirement
that any two approximate $1$-characteristics cannot cross each other.

\subsection*{Case V}($\alphah < 0, \betah < 0, \gammah > 0$):
In this case
\begin{gather*}
E_1^+(\dmh) = 0, \quad E_1^-(\dmh) = \alphah + \betah, \\
L_1^+(\dmh) = \gammah, \quad L_1^-(\dmh) = 0, \\
C_1(\dmh) = 0, \quad \Delta_1(\dmh ) = \gammah + |\alphah| + |\betah|.
\end{gather*}
One may choose the continuation line segment as the leftmost characteristic of $\gammah$, then
\begin{gather*}
E_1^\pm (\dml) = 0, \quad L_1^+(\dml) = 0, \quad S_1(\dml) = 0,\\
C_1^\pm (\dml) = 0, \quad \Delta_1(\dml) = 0
\end{gather*}
and 
\begin{gather*}
E_1^+(\dmr) = 0, \quad E_1^-(\dmr) = \betah, \\
L_1^+(\dmr) = \gammah, \quad S_1(\dmr) = 0, \\
C_1^\pm(\dmr) = 0, \quad \Delta_1(\dmr) = \gammah + |\betah|.
\end{gather*}
Now it is direct to check \eqref{4.*1}--\eqref{4.**3}.
Moreover, combining the selection in this case with its symmetric counterpart, one can check that any two approximate $1$-characteristics would not cross each other in this case.

\subsection*{Case VI}($\alphah < 0, \betah>0, \gammah>0$):
In this case
\begin{gather*}
E_1^+(\dmh) = \betah, \quad E_1^-(\dmh) = \alphah, \\
L_1^+(\dmh) = \gammah, \quad L_1^-(\dmh) = 0, \\
C_1(\dmh) = \min \{ - \alphah, \betah \}, \quad \Delta_1(\dmh) = \big| \gammah + |\alphah| - \betah \big|.
\end{gather*}
One may choose the continuation line segment as the leftmost characteristic line of $\gammah$, then
\begin{gather*}
E_1^\pm(\dml) = 0, \quad L_1^+(\dml) = S_1(\dml) = 0, \\
C_1^\pm(\dml) = 0, \quad \Delta_1(\dml) = 0,
\end{gather*}
and
\begin{gather*}
E_1^+(\dmr) = \betah, \quad E_1^-(\dmr) = 0, \\
L_1^+(\dmr) = \gammah, \quad S_1(\dmr) = 0, \\
C_1^+(\dmr) = C_1(\dmh), \quad C_1^-(\dmr) = 0, \quad \Delta_1(\dmr) = \Delta_1(\dmh).
\end{gather*}
So that \eqref{4.*1}--\eqref{4.**3} hold.

\subsection*{Case VII}($\alphah = \alphal + \alphar > 0, \betah < 0, \gammah > 0$):
In this case
\begin{gather*}
E_1^+(\dmh) = \alphah, \quad E_1^-(\dmh) = \betah, \\
L_1^+(\dmh) = \gammah, \quad L_1^-(\dmh) = 0, \\
C_1(\dmh) = \min \{ \alphah, -\betah \}, \quad \Delta_1(\dmh) = \big| \gammah - \alphah + |\betah| \big|
\end{gather*}
and 
\begin{gather*}
E_1^+(\dml) = \alphal, \quad E_1^-(\dml) = 0,\\
E_1^+(\dmr) = \alphar, \quad E_1^-(\dmr) = \betah.
\end{gather*}
\subsubsection*{Subcase VII.I}($|\betah| > \alphah$):
One my choose the continuation line segment as the leftmost characteristic line of $\gammah$, then
\begin{gather*}
L_1^+(\dml) = S_1(\dml) = 0, \quad C_1^+(\dml) = \alphal, \\
C_1^-(\dml) = 0, \quad \Delta_1(\dml) = 0, \\
\end{gather*}
and
\begin{gather*}
L_1^+(\dmr) = \gammah, \quad S_1(\dmr) = 0, \\
C_1^+(\dmr) = \alphar, \quad C_1^-(\dmr) = \alphah, \quad \Delta_1(\dmr) = \Delta_1(\dmh).
\end{gather*}
Due to the fact $C_1(\dmh) = \alphah, \Delta_1(\dmh) > \gammah$ in this situation, it is easy to check \eqref{4.*1}--\eqref{4.**3}.
\subsubsection*{Subcase VII.II}($\alphah \geq |\betah| > \alphar$):
The continuation line segment can be chosen as the rightmost characteristic line of $\gammah$ if $\gammah < \alphah - |\betah|$ and as the characteristic line that the leaving $1$-rarefaction wave on its left is of strength $\alphah - |\betah| $ if $\gammah \geq \alphah - |\betah|$. 
Then
\begin{gather*}
L_1^+(\dml) = \min\{ \gammah, \alphah - |\betah| \}, \quad S_1(\dml) = 0, \\
C_1^+(\dml) = |\betah| - \alphar, \quad C_1^-(\dml) = 0, \\
\Delta_1(\dml) = \max \{ 0, \alphah - |\betah| - \gammah  \}
\end{gather*}
and 
\begin{gather*}
L_1^+(\dmr) = \max \{ \gammah - \alphah + |\betah|, 0 \}, \quad S_1(\dmr) = 0, \\
C_1^+(\dmr) = \alphar, \quad C_1^-(\dmr) = |\betah|, \\
\Delta_1(\dmr) = \max \{ 0, \gammah - \alphah + |\betah| \}
\end{gather*}
Noting that in this situation, $C_1(\dmh) = |\betah|$, one can show \eqref{4.*1}--\eqref{4.**3} easily.
\subsubsection*{Subcase VII.III}($\alphar \geq |\betah|$):
One may choose the continuation line segment as the right most characteristic line of $\gammah$ if $\gammah \leq \alphal$ and the characteristic line that the leaving $1$-rarefaction wave on its left is of strength $\alphal$ if $\gammah > \alphal$. 
Then
\begin{gather*}
L_1^+(\dml) = \min \{ \gammah, \alphal \}, \quad S_1(\dml) = 0, \\
C_1^\pm (\dml) = 0, \quad \Delta_1(\dml) = \max \{ 0, \alphal - \gammah \}
\end{gather*}
and 
\begin{gather*}
L_1^+(\dmr) = \max \{ \gammah - \alphal,0 \}, \quad S_1(\dmr) = 0, \\
C_1^\pm (\dmr) = |\betah|, \quad \Delta_1(\dmr) = \Delta_1(\dmh) - \Delta_1(\dml). 
\end{gather*}
Noting that in this situation, $C_1(\dmh) = |\betah|$ and when $\gammah \leq \alphal$, it holds that $\Delta_1(\dmr) = \alphar - |\betah|$,  so \eqref{4.*1}--\eqref{4.**3} are valid.

Moreover, combining Cases VI--VII and their symmetric counterparts, 
one can show that any two approximate $1$-characteristics would not cross in this case.

\subsection*{Case VIII}($\alphah = \alphal + \alphar > 0, \betah > 0, \gammah > 0$):
In this case
\begin{gather*}
E_1^+(\dmh) = \alphah + \betah, \quad E_1^-(\dmh) = 0, \\
L_1^+(\dmh) = \gammah, \quad L_1^-(\dmh) = 0, \\
C_1(\dmh) = 0, \quad  \Delta_1(\dmh) = |\gammah - \alphah - \betah|
\end{gather*}
and
\begin{gather*}
E_1^+(\dml) = \alphal, \quad E_1^-(\dml)=0, \\
S_1(\dml) = 0, \quad  C_1^\pm(\dml) = 0, \\
E_1^+(\dmr) = \alphar + \betah, \quad E_1^-(\dmr) = 0, \\
S_1(\dmr) = 0, \quad C_1^\pm(\dmr) = 0.
\end{gather*}
When $\Delta_1(\dmh)$ strengthens the $1$-rarefaction waves, one may locate its effect at the center part of $\gammah$, and when it weakens them, one may divide it into two parts whose effects on $\alphah$ and $\betah$ are proportional to its original strength. The continuation line segment can be chosen accordingly as what follows.
\subsubsection*{Subcase VIII.I}($ \gammah > \alphal + \frac{\alphal}{\alphah} \betah$):
The continuation line segment is chosen as the $1$-characteristics line that the strength of $1$-rarefaction waves on its left is $\alphal$. 
Then
\begin{gather*}
L_1^+(\dml) = \alphal, \quad \Delta_1(\dml) = 0, \\
L_1^+(\dmr) = \gammah - \alphal, \quad \Delta_1(\dmr) = \Delta_1(\dmh).
\end{gather*}
\subsubsection*{Subcase VIII.II}($ \gammah \leq \alphal + \frac{\alphal}{\alphah} \betah$):
The continuation line segment is chosen as the $1$-characteristic line that the strength of the leaving $1$-rarefaction wave on its left is $\frac{\alphah}{\alphah+\betah} \gammah$, and
\begin{gather*}
L_1^+(\dml) = \frac{\alphah}{\alphah + \betah} \gammah, \quad \Delta_1(\dml) = \alphal - \frac{\alphah}{\alphah + \betah} \gammah, \\
L_1^+(\dmr) = \frac{\betah}{\alphah + \betah } \gammah, \quad \Delta_1(\dmr) = \alphar + \betah - \frac{\betah}{\alphah + \betah} \gammah.
\end{gather*}
Now one can check directly that \eqref{4.*1}--\eqref{4.**3} hold and any two approximate $1$-characteristics would not cross each other in this case.

\bibliographystyle{plain}

\bibliography{WNGD}

%
%
%
%
%
%
%
%

\end{document}